\documentclass[10pt]{article}
\newcommand{\ddate}{July 4 2018}

\date{\ddate}
\usepackage[matrix,arrow,curve]{xy}
\usepackage{amssymb,amsthm}
\usepackage{graphicx}
\usepackage{pstricks}
\usepackage{psfrag}

\usepackage{hyperref}

\usepackage{epic}
\usepackage{eepic}

\newtheorem{dummy}{anything}[section]
\newtheorem{Theorem}[dummy]{Theorem}
\newtheorem{Lemma}[dummy]{Lemma}
\newtheorem{Proposition}[dummy]{Proposition}
\newtheorem{Corollary}[dummy]{Corollary}

\newtheorem{Example}[dummy]{Example}
\newtheorem{Examples}[dummy]{Examples}
\newtheorem{Remark}[dummy]{Remark}
\newtheorem{Remarks}[dummy]{Remarks}
\newtheorem{Problem}[dummy]{Problem}
\newtheorem{Question}[dummy]{Question}
\newtheorem{Conjecture}[dummy]{Conjecture}
\newtheorem{ccote}[dummy]{}
\newtheorem*{thmA}{Theorem A}
\newtheorem*{thmB}{Theorem B}
\newtheorem*{thmC}{Theorem C}
\newtheorem*{thmD}{Theorem D}

\newtheorem*{corE}{Corollary E}

\newcommand{\bbr}{{\mathbb R}}

\newcommand{\bbc}{{\mathbb C}}
\newcommand{\bbz}{{\mathbb Z}}

\newcommand{\bbn}{{\mathbb N}}

\newcommand{\cala}{{\mathcal A}}
\newcommand{\calb}{{\mathcal B}}
\newcommand{\calc}{{\mathcal C}}
\newcommand{\cald}{{\mathcal D}}

\newcommand{\cali}{{\mathcal I}}

\newcommand{\calm}{{\mathcal M}}
\newcommand{\caln}{{\mathcal N}}

\newcommand{\calt}{{\mathcal T}}

\newcommand{\mancqfd}{\hfill \ensuremath{\Box}}

\newcommand{\pcirc}{\kern .7pt {\scriptstyle \circ} \kern 1pt}
\newcommand{\tpcirc}{\kern .7pt {\scriptstyle \tilde\circ} \kern 1pt}
\newcommand{\mun}{{-1}}
\newcommand{\scr}{\scriptscriptstyle}

\renewcommand{\:}{\colon}

\newcommand{\beq}[1]{\begin{equation}\label{#1}}
\newcommand{\eeq}{\end{equation}}

\newcommand{\sk}[1]{\vskip #1 mm}
\newcommand{\eqref}[1]{(\ref{#1})}
\newcommand{\hfl}[2]{\smash{\mathop{\hbox to 1 truecm{\kern %
3pt\rightarrowfill\kern 3pt}}%
\limits^{\scriptstyle#1}_{\scriptstyle#2}}}
\newcommand{\cqfd}{\unskip\kern 6pt\penalty 500%
\raise -2pt\hbox{\vrule\vbox to10pt{\hrule width %
4pt\vfill\hrule}\vrule}\smallskip}
\newcommand{\proref}[1]{Proposition~\ref{#1}}
\newcommand{\remref}[1]{Remark~\ref{#1}}
\newcommand{\lemref}[1]{Lemma~\ref{#1}}
\newcommand{\corref}[1]{Corollary~\ref{#1}}
\newcommand{\thref}[1]{Theorem~\ref{#1}}
\newcommand{\exref}[1]{Example~\ref{#1}}

\newcommand{\secref}[1]{Section~\ref{#1}}

\newcommand{\dfn}[1]{{\it #1}}

\newcommand{\dia}[1]{\begin{array}{c}{\xymatrix@C-3pt@M+2pt@R-4pt{#1 }}\end{array}}

\newcommand{\fl}[1]{\buildrel{#1}\over{\longrightarrow}}

\newcommand{\di}{\approx_{\rm diff}}
\newcommand{\ho}{\approx_{\rm top}}
\newcommand{\sdi}{\approx_{\,{\bbr}\!\mbox{-}\rm diff}}

\newcommand{\kdi}{\approx_{\,P\!\mbox{-}\rm diff}}
\newcommand{\sundi}{\approx_{\,S^1\!\mbox{-}\rm diff}}
\newcommand{\rel}{{\rm rel\,}}
\newcommand{\wh}{{\rm Wh}}

\newcommand{\cob}{{\rm Cob}}

\newcommand{\bbun}{{\mathbf 1}}
\newcommand{\rp}[1]{\bbr P^{#1}}

\newcommand{\sdiff}{{\rm Diff}_{\bbr}}

\newcommand{\aut }{{\rm Diff}}
\newcommand{\Di }{{\rm Diff}}
\newcommand{\id}{{\rm id}}

\newcommand{\dcup}{\amalg}

\title{A simplification problem in manifold theory}
\author{Jean-Claude HAUSMANN and Bj\o rn JAHREN}

\begin{document}
\maketitle

\begin{abstract}Two smooth manifolds M and N are called $\bbr$-diffeomorphic if $M\times\bbr$ is
diffeomorphic to $N\times\bbr$. We consider the following simplification problem:
{\it does $\bbr$-diffeomorphism imply diffeomorphism or homeomorphism?} For compact manifolds, 
analysis of this problem relies on some of the main achievements
of the theory of manifolds, in particular the h- and s-cobordism theorems
in high dimensions and the spectacular more recent classification results
in dimensions 3 and 4. This paper presents what is currently known about
the subject as well as some new results about classifications
of $\bbr$-diffeomorphisms. 
\end{abstract}

{\small
\tableofcontents
}

\section{Introduction}

Let $X$ and $Y$ be smooth manifolds.
We write $Y\di X$ when $Y$ is diffeomorphic to $X$ and  $Y\ho X$ when $Y$ is homeomorphic to $X$.
Given a manifold $P$, $Y$ and $X$ are called \dfn{$P$-diffeomorphic} (notation: $Y \kdi X$) if there exists 
a diffeomorphism $f\: Y\times P \to X\times P$, and such an $f$ is called a \dfn{$P$-diffeomorphism}. 
Consider the following {\it simplification problem}.

\sk{1}\noindent
{\bf The $P$-Simplification Problem. \ } \em For smooth closed manifolds, does $P$-diffeomorphism imply diffeomorphism, or homeomorphism? \rm
\sk{1}

The first part of this paper is a survey on what is currently known about the $\bbr$-simplification problem
(other cases are briefly discussed in \secref{S.Rem}).
This quite natural question, expressed in very elementary terms, happens to be
closely related to the theory of invertible 
cobordisms (see e.\,g. \cite{Stallings, JaKwaRP} and \proref{P.invcob}).
As advertisement, here are some samples of the main results of the theory.

\begin{thmA}
Let $M$ and $N$ be smooth closed manifolds of dimension $n$. Suppose that $M$ is simply connected.
Then
\begin{itemize}
\item[(i)]  $N\sdi M \ \Longrightarrow \ N\ho M$.
\item[(ii)]  $N\sdi M \ \Longrightarrow \ N\di M$ if $n\neq 4$.
\end{itemize}
\end{thmA}

The simplicity of the statement of Theorem~A, with almost no dimension restriction, contrasts with
the variety of techniques involved in the proof. Actually, Theorem~A concentrates 
a good deal of important developments in differential topology during the 20th century (see also Section~\ref{hinote}). 

When $M$ is not simply connected, part~(i) of Theorem~A is false in general,
The first counterexample was essentially given by Milnor in a famous paper 
in 1961 \cite{MilnorTwoCom}  (see \exref{IneqWh}.(1)). 
Using a recent result of Jahren-Kwasik \cite[Theorem~1.2]{JaKwa}, we 
now  know that  part~(i) is, in general, ``infinitely false'', i.e.
there are manifolds having countably many homeomorphism classes within their 
$\bbr$-diffeomorphism class (see \exref{IneqWh}.(5)).

In dimension 4, part~(ii) of Theorem~A is infinitely false in general, even when $M$ is
simply connected. Indeed, there may be a countable infinity of diffeomorphism classes of manifolds
within the homeomorphism class of M, for instance  when $M=\bbc P^2 \,\sharp\,k\,\overline{\bbc P^2}$.
the connected sum of the complex projective space $\bbc P^2$ and $k$ copies of $\bbc P^2$ with 
reversed orientation, $k\geq 6$ \cite{FinStern}. Each such diffeomorphism class provides a 
counterexample of part~(ii) of Theorem~A, thanks to the following result (probably known by specialists). 

\begin{thmB}
Let $M$ and $N$ be smooth closed  manifolds of dimension 4 which are 
homeomorphic.  Suppose that $H_1(M,\bbz_2)=0$. Then $N\sdi M$.
\end{thmB}

In particular, although it is not known whether 
all differentiable structures on the 4-sphere $S^4$ are diffeomorphic (the smooth,
4-dimensional Poincare conjecture), they would  all be $\bbr$-diffeomorphic. Incidentally,
the possibility of such exotic structures will play a role in some results in
Sections~\ref{S.nfour}, \ref{S.nthree} and~\ref{S.classif}.

Note also that manifolds M and N as in Theorem~B but simply-connected are
homeomorphic if and only they are homotopy equivalent \cite[\S~10.1]{Freedman-Quinn-Book}. 

The hypothesis of simple connectivity in Theorem~A is not necessary in low dimensions.
The following result is classical for $n\leq 2$, follows for $n=3$ from a
result of Turaev \cite{Turaev2} together with the geometrization theorem.

\begin{thmC}
Let $M$ and $N$ be two closed manifolds of dimension $n\leq 3$, which are orientable if $n=3$. 
Then $N\sdi M$ if and only if $N\di M$.
\end{thmC}

Theorem~C is currently unknown for non-orientable $3$-manifolds (see \remref{nonorthreemfds}).

Proofs of Theorems A, B and C are given in Sections~4--6 (with more general 
hypotheses for Theorem~A), after important preliminaries in  Sections~2--3.
Of particular importance for the simplification problem are the so-called
{\em inertial} invertible cobordisms, characterized  by the property that the two ends 
are diffeomorphic (homeomorphic). Section 4 also includes some new results 
in this area (notably Proposition \ref{P.sigbar}).

In the last part of this paper (\secref{S.classif}), we present new results on
classification of $\bbr$-diffeomorphisms under several equivalence relations. For instance, 
a diffeomorphism $f\:N\times\bbr\to M\times\bbr$ is called \dfn{decomposable} 
if there exists a diffeomorphism $\varphi\:N\to M$ such that $f$ is isotopic to 
$\varphi\times{\rm \pm id}_\bbr$. 
Fix a manifold $M$ and 
consider pairs $(N,f)$ where $N$ is a smooth closed manifold and 
$f\:N\times\bbr\to M\times\bbr$ is a diffeomorphism. Two such pairs $(N,f)$ and $(\hat N,\hat f)$
are equivalent if $f^\mun\pcirc\hat f$ is decomposable. 
The set of equivalence classes is denoted by $\cald(M)$. We compute this set in all dimensions
in terms of invertible cobordisms. As a consequence, in high dimensions  we get
the following result.

\begin{thmD}
Let $M$ be a closed connected smooth manifold of dimension $n\geq 5$. Then 
$\cald(M)$ is in bijection with the Whitehead group $\wh(\pi_1M)$.
\end{thmD}

\begin{corE}
Let $M$ be a closed connected smooth manifold of dimension $n\geq 5$. The following assertions are equivalent.
\begin{itemize}
\item[(i)]   $\wh(\pi_1M)=0$. 
\item[(ii)]  Any diffeomorphism $f\:N\times\bbr\to M\times\bbr$ is decomposable. 
\end{itemize}
\end{corE}

Theorem D is actually a consequence of a more categorical statement 
(Theorem \ref{T.ic0}), which is of independent interest.

We also consider a quotient $\cald_c(M)$ of $\cald(M)$ where isotopy is replaced by concordance.
Interesting examples are produced to discuss the principle of {\it concordance implies isotopy}
for $\bbr$-diffeomorphisms.

\sk{3}\noindent{\bf Acknowledgments: }\rm The first author thanks 
Pierre de la Harpe and Claude Weber for useful discussions. 
The second author thanks S\l awomir Kwasik for numerous enlightening 
discussion about material related to this paper. We are also grateful to Matthias Kreck and Raphael Torres 
for useful comments and to the referee for a careful reading of the first version of this paper.

\section{Cobordisms}\label{S.cob}

\begin{ccote}\label{co.preli} Preliminaries. \ \rm
Throughout this paper, we work in the smooth category $\calc^\infty$ of smooth manifolds, 
(possibly with corners: see below) and smooth maps. Our manifolds are not necessarily orientable.

If $X$ is a manifold and $r\in\bbr$, the formula $j_X^r(x)=(x,r)$ defines 
a diffeomorphism $j_X^r\:X\to X\times\{r\}$
or an embedding $j_X^r\:X\to X\times\bbr$, depending on the context.

Let $X$ and $X'$ be manifolds with given submanifolds $Y$ and $Y'$, resp., 
and let $\imath\:Y\fl{\approx}Y'$ be an identification (diffeomorphism),
usually understood from the context.
A map $f\:X\to X'$ is called \dfn{relative $Y$} (notation: $\rel\ Y$) if 
the restriction of $f$ to $Y$ coincides with the identification~$\imath$.
Often, $Y=\partial X$ and $Y'=\partial X'$, 
in which case we say \dfn{relative boundary} (notation: $\rel \partial$).
\end{ccote}

\begin{ccote}\label{Co.cobordism} The cobordism category. \rm \
A \dfn{triad} is a triple $(W,M,N)$ of compact smooth manifolds
such that $\partial W = (M\dcup N) \cup X$ with $X\di \partial M \times I$.
Most often $\partial M$ is empty, in which case $\partial W = M\dcup N$.
Otherwise, $W$ is actually a manifold with {\em corners 
 along $\partial M$ and $\partial N$}, modeled locally on the subset 
$\{(x_1,\dots,x_n)|x_1\geq 0, x_2\geq 0\}$ of $\bbr^n$. Smooth maps
are then always required to preserve the stratification coming from this 
local structure (for a precise exposition of the smooth category with corners, see the appendix of \cite{BorelSerreDouady}).

Let us fix the manifolds $M$ and $N$ (one or both of them could be empty). 
 A \dfn{cobordism} from $M$ to $N$ is a triple 
$(W,j_M,j_N)$, where $W$ is a compact smooth manifold and  
$j_M:M\to\partial W$, $j_N:N\to \partial W$ are 
embeddings such that $(W,j_M(M), j_N(N))$ is a triad. If $M$ and $N$ have
nonempty boundaries, $(W,j_M,j_N)$ will sometimes be called a {\em relative
cobordism}.

By a slight abuse of notation we will also let $ j_M$ denote the embedding $j_M$ 
considered as a map into $W$.

Two cobordisms  $(W,j_M,j_N)$ and $(W',j'_M,j'_N)$ are \dfn{equivalent} if 
there is a diffeomorphism $h\:W\to W'$ such that $ j_M\pcirc h=  j'_M$ and
 $ j_N\pcirc h= j'_N$. The set of equivalence classes of cobordisms from 
$M$ to $N$ is denoted by $\cob(M,N)$.
The equivalence class of $(W,j_M,j_N)$ is denoted by $[W,j_M,j_N]$.  

A triad $(W,M,N)$ determines an obvious cobordism, $(W,\imath_M,\imath_N)$, and
its equivalence class in $\cob(M,N)$ will also be denoted by $[W,M,N]$.
Note that $[W,M,N]=[W',M,N]$ if and only if $W\di W'\ (\rel M\cup N)$.
We shall make no distinction between a triad and the cobordism it determines and
often  write ``a cobordism $(W,M,N)$'' instead of ``a triad $(W,M,N)$''. 
A  triad 
of the form $(M\times I,M\times \{0\},M\times \{1\})=(M\times I,j_M^0, j_N^1)$ 
(using the notations $j_X^r$ from Section~\ref{co.preli}) will be 
called a {\em trivial} cobordism.

We now define a \dfn{composition}
$$
\cob(M,N)\times \cob(N,P) \fl{\pcirc} \cob(M,P) \, .
$$
Let $c\in\cob(M,N)$ and $c'\in\cob(N,P)$, represented by cobordisms 
$(W,j_M,j_N)$ and $(W',j'_N,j'_P)$.
The topological manifold $W\cup_{j'_N\pcirc j_N^\mun}W'$ admits a smooth structure 
compatible with those on $W$ and $W'$
\cite[Theorem~1.4]{MilnorHCob}. Such a smooth structure is unique up to 
diffeomorphism relative boundary
(see also \cite[Chapter~8, \S\,2]{Hirsch}).
Choosing one of these smooth structures gives rise to a smooth manifold 
$W\pcirc W'$, and $(W\pcirc W',j_M, j'_P)$  
represents a well-defined class $c\pcirc c'\in\cob(M,P)$. With this composition, 
one gets a category $\cob$ whose objects are closed smooth manifolds and whose set 
of morphisms from $M$ to $N$ is $\cob(M,N)$.  The identity at the object $M$ is 
represented by the trivial cobordism:
$$
\bbun_M=[M\times I,M\times \{0\},M\times \{1\}]=[M\times I,j_M^0, j_N^1] \, .
$$ 

 Note that, by construction,  the composition 
$\bbun_M\pcirc(W, j_M, j_N)\pcirc\bbun_N$ has the form of a triad $(W',M,N)$, 
where we identify $M$ and $N$ with $M\times \{0\}$ and $N\times \{1\}$.
In other words: up to equivalence, cobordisms can always be represented by 
triads.  This will sometimes be exploited in proofs, in order to simplify
notation.  But in general it is helpful  to have  the extra flexibility of the
more general definition, as it makes it easier to keep track of how we 
identify $M$ and $N$ with submanifolds of $\partial W$. A trivial example is
$\bbun_M$, which as a cobordism goes from $M$ to itself, but in a triad the 
two ends can not be the same manifold.  More examples are 
the definition of mapping cylinders and \lemref{L.mapscob1} below.

Our definition of the cobordism category is a condensed reformulation of 
\cite[\S~1]{MilnorHCob},
with end-identifications going in reverse directions.
\end{ccote}

\begin{ccote} Duals and mapping cylinders.\rm \
The order of $M$ and $N$ in $(W,j_M, j_N)$ reflects the categorical intuition
 that $W$ is a cobordism {\em from $M$ to $N$}.  Reversing the order of 
$M$ and $N$, we obtain the 
 \dfn{dual cobordism} $(\bar W,j_N, j_M)$, where $\bar W$ is just a copy of $W$.
If the cobordism is given by a triad $(W,M,N)$, its dual is given by $(\bar W,N,M)$.
The correspondence $[W]\to [\bar W]$ defines a functor $\cob\to\cob^{\rm op}$ which is an isomorphism of categories.

Examples of cobordisms are given by mapping cylinders of diffeomorphisms. 
Let $f\:M\to N$ be a diffeomorphism between smooth closed manifolds. The mapping cylinder $C_f$ of $f$ is defined
by
\beq{defCf}
C_f = \big\{M\times [0,1)\big\}  \cup \big\{N \times (0,1]\big\} \big/ \{(x,t)\sim (f(x),t) 
\hbox{ for all } (x,t)\in M\times (0,1)\} \, . 
\eeq
Note the obvious homeomorphism 
\beq{defCfTop}
C_f\ho \big\{M\times I \cup N\big\} \big/\{(x,1)\sim f(x)\} \, .
\eeq
The latter is the usual definition of the mapping cylinder valid for any continuous map~$f$.
But, when $f$ is a diffeomorphism, 
Definition~\eqref{defCf} makes $C_f$ a smooth manifold with boundary
$\partial C_f=M\times\{0\}\cup N\times\{1\}$. We thus get a cobordism 
$(C_f,j_M^0, j_N^1)$.

\begin{Lemma}\label{L.mapscob1}
For a diffeomorphism $f\:M\to N$ between smooth closed manifolds, the equalities
\beq{L.mapscob1-eq}
[C_f,j_M^0, j_N^1] = [M\times I,j_M^0, j_M^1\pcirc f^\mun] = 
[N\times I,j_N^0\pcirc f , j_N^1]
\eeq
hold in $\cob(M,N)$.
\end{Lemma}

\begin{proof}
 One checks that the correspondences
\beq{E01.mapscob}
\left\{
\begin{array}{rcl} 
M\times [0,1) \ni (x,t) & \mapsto &  (x,t) \\
N\times (0,1] \ni (y,t) & \mapsto &  (f^\mun(y),t)  \ .
\end{array}
\right.
\eeq
provide the first equality. The second one is obtained similarly. 
\end{proof}

\begin{Example}\label{ex.conc}\rm Let $f:M\to M$ be a self-diffeomorphism of a closed manifold $M$. 
Then $C_f$ is equivalent to $\bbun_M$ if and only if there is a diffeomorphism
$F:M\times I\to M\times I$ such that $F(x,0)=f(x)$ and $F(x,1)=x$, i.\,e.
$F$ is {\rm concordant} to $\id_M$.
\end{Example}

The proof of the following lemma is left to the reader (compare \cite[Theorems~1.6]{MilnorHCob}).

\begin{Lemma}\label{L.mapcyl} Let $M\fl{f} N \fl{g} P$ be diffeomorphisms between smooth manifolds. Then 
$[C_{g\pcirc f}]= [C_f] \pcirc  [C_g]$.  \mancqfd
\end{Lemma}

\noindent {\em Remark.} The reason for the contravariant form of this identity 
is that we write
composition of cobordisms ``from left to right''. This is the usual convention
in cobordism categories, like path categories (e.\,g. fundamental groupoid)
and topological field theories. 
\end{ccote}

\section{Invertible cobordisms}\label{S.icob}

\begin{ccote}\label{invCob} The category of invertible cobordisms. \rm \  
A cobordism $(W,j_M,j_N)$ is called \dfn{invertible} if $[W]$ is an invertible morphism in $\cob$,
i.\,e. there is a cobordism 
$(W^\mun,j_N,j_M)$ such that $[W]\pcirc [W^\mun] = \bbun_M$ and $[W^\mun]\pcirc [W] = \bbun_N$.

As usual, these conditions uniquely determine $[W^\mun]$ if it exists. 
Two smooth manifolds are \dfn{invertibly cobordant} if there exists an invertible cobordism between them.
Let $\cob^*(M,N)$ be the subset of $\cob(M,N)$ formed by invertible cobordisms. This defines a subcategory
$\cob^*$ of $\cob$, with the same objects.

An example of invertible cobordism is given by the mapping cylinder $C_f$ of a diffeomorphism $f\:N\to M$.
Indeed, \lemref{L.mapcyl} together with \lemref{L.mapscob1} imply that 
$[C_f]^\mun = [C_{f^\mun}]= [\overline{C_f}]$.
\end{ccote}

\begin{ccote} Invertible cobordisms and $\bbr$-diffeomorphisms. \rm \ 
From now on until Section~7 we will be mainly concerned with cobordisms 
between {\em closed} manifolds, unless explicitly stated.  The main
exceptions are the discussions of h-cobordism and Whitehead torsion in
Sections~\ref{Co.hcob} and \ref{Co.tau} and of concordance in 
Section~\ref{concordance}.

Here is one of the main results of this section.

\begin{Proposition}\label{P.invcob}
Let $M$ and $N$ be smooth closed manifolds.
The following statements are equivalent.
\begin{itemize}
\item[(a)] $N\sdi M$.
\item[(b)] $N$ and $M$ are invertibly cobordant.
\item[(c)] There is a diffeomorphism $\beta\: N\times S^1 \to M\times S^1$ such that the composed homomorphism
{\small
\beq{comptriv}
\dia{
\pi_1(N\times pt) \ar[r] &  \pi_1( N\times S^1)  \ar[r]^{\beta_*} & \pi_1( M\times S^1) \ar[r]^(.6){proj} & \pi_1(S^1)
}
\eeq
}
is trivial.
\item[(d)] There is a diffeomorphism $\beta\: N\times S^1 \to M\times S^1$ such that the diagram
\beq{piUnComm}
\dia{
\pi_1( N\times S^1)  \ar[rr]^{\beta_*} \ar[rd]  && 
\pi_1( M\times S^1) \ar[ld] \\
& \pi_1(S^1)
}
\eeq
commutes, where the arrows to $\pi_1(S^1)$ are induced by the projections onto~$S^1$
\end{itemize}
\end{Proposition}

\begin{Remark}\label{R.Sunsimpl}\rm
Conditions (c) or (d) are stronger than just $S^1$-diffeomorphism, since there are examples of 
closed manifolds $M$ and $N$ such that $M\sundi N$ but $\pi_1(N)\not\approx\pi_1(M)$
(see e.g. \cite[p.~29]{charlap}, \cite[Theorem~4.1]{ConRay} or \cite[Theorem~2]{KwaRok}). 
Some of these examples are in dimension 3, so crossing with
spheres provide examples in all dimensions greater than four.
\end{Remark}

We write a detailed proof of \proref{P.invcob}, introducing notations which will be useful in 
\secref{S.classif}. Also, proving $(a)\Rightarrow (c)$ is delicate: Kervaire wrote a short argument 
at the end of \cite{KeBMS} but, after publication, thought that his argument was incorrect.
For a proof of $(b)\Rightarrow (c)$ using the deep s-cobordism theorem, when $\dim M \geq 4$, see \remref{MxS1usingscob}.

\begin{proof}[Proof of \proref{P.invcob}] \
{\em (a) implies (b).} \  Let $f\: N\times\bbr \to M\times\bbr$ be a diffeomorphism. 
Write $M_u=M\times\{u\}$, $N_u=N\times\{u\}$ and $N'_u=f(N_u)$. We use the obvious diffeomorphisms
$j_M^u\:M\to M_u$ and $j_N^u\:N\to N_u$ introduced in \secref{co.preli}.

By compactness of $N$, there exists $r<u<s<v$ such that  $N_u'\subset M\times (r,s)$ and $M_s\subset f(N\times (u,v))$
(to get this order, one might have to precompose $f$ by the automorphism $(x,u)\mapsto (x,-u)$ of $N\times\bbr$).
The region $A$ between $M_r$ and $N'_u$ and the region $B$ between $N'_u$ and $M_s$ produce equivalence classes of cobordisms 
$$
[A,j_M^r, f\pcirc j_N^u] \in \cob(M,N) \ , \
[B,f\pcirc j_N^u, j_M^s] \in \cob(N,M) 
$$
obviously satisfying $[A]\pcirc [B]= \bbun_M$. One also has the class of cobordism
$$
[A',j_M^s, f\pcirc j_N^v] \in \cob(M,N) \, .
$$
Using the diffeomorphism $f$, one proves that $[B]\pcirc [A']= \bbun_N$. This implies that $[A']=[A]$ and 
$[B]=[A]^\mun$.

\sk{2}
\noindent {\em (b) implies (a) and (c).} \ We first prove that (b) implies (a), using an argument of Stallings \cite[\S\,2]{Stallings}.
Let $A$ be an invertible cobordism from $M$ to $N$, with inverse~$B$. 
Let $A_i$ and $B_i$ be copies of $A$ and $B$ indexed by $i\in\bbz$. 
Consider the manifold 
\beq{E.invcob}
\begin{array}{rcl}
W &=& \cdots\pcirc\, (A_i\pcirc B_i)\, \pcirc \,(A_{i+1}\pcirc B_{i+1}) \,\pcirc \cdots \\ &=& 
\cdots\pcirc\, (B_i\pcirc A_{i+1})\,\pcirc\,  (B_{i+1}\,\pcirc A_{i+2})        \pcirc\cdots
\end{array}
\eeq

Let $g_i\: M\times [i,i+1] \to A_i\pcirc B_i$ be copies of some diffeomorphism relative boundary $g\: M\times I\to A\pcirc B$.
Then, $g_{\scr M}=\bigcup_{i\in\bbz} g_{i}$ is a diffeomorphism from $M\times\bbr$ onto $W$. 
The same may be done with the second decomposition of $W$. We thus get two
diffeomorphisms  $g_{\scr M}\: M\times\bbr\to W$ and $g_{\scr N}\: N\times\bbr\to W$, which proves (a). 

We now prove that (b) implies (c). By conjugation by $g_M$, the automorphism 
$(x,t)\to (x,t+1)$ of $M\times\bbr$ 
produces an automorphism
$T$ of $W$, generating a free and proper $\bbz$-action on $W$
and a diffeomorphism $\alpha\: W/\bbz \fl{\approx}  M\times S^1$. 
It is not clear whether the corresponding automorphism obtained via $g_{\scr N}$ is conjugate to $T$. 
However, the manifold
$Z_i= B_i\pcirc A_{i+1}$ is a fundamental domain for the $T$-action and the restriction of $T$ to
$Z_i$ sends $Z_i$ onto $Z_{i+1}$ relative boundary. 
Therefore, we get a diffeomorphism
$$
\dia{ 
\beta : N\times S^1 \ar[r]_{\di} &
 N\times(I/\partial I) \ar[r]_(0.63){\di} &  W/\bbz \ar[r]_(.40){\di}^(.40){\alpha}  & M\times S^1  \, .
}
$$
The composed homomorphism \eqref{comptriv} is trivial since the restriction of $\beta$ to $N\times pt$ factors through $M\times\bbr$.

\sk{2}
\noindent {\em (c) implies (d).} \ Using the exact sequence
$$
1 \to \pi_1(N\times pt) \to \pi_1( N\times S^1) \fl{proj} \pi_1(S^1)  \to 1
$$
Condition (d) implies that $proj\pcirc\beta_*$ factors
through an endomorphism $\bar\beta_*$ of $\pi_1(S^1)$ which, being surjective, satisfies $\bar\beta_*(b)=\pm b$
(identifying $\pi_1(S^1)$ with $\bbz$).
The possible negative sign may be avoided by precomposing $\beta$ 
with the automorphism $(x,z)\mapsto (x,\bar z)$ of $N\times S^1$.

\sk{2}
\noindent {\em (d) implies (a).} \  
Let $\beta\: N\times S^1\to M\times S^1$ as in~(d). 
 Consider the pullback diagram
$$
\dia{
P \ar[r]^{\tilde\beta} \ar[d]^{p} &  M\times \bbr \ar[d]^{proj} \\
N\times S^1  \ar[r]^\beta & M\times S^1
}
$$
The map $\tilde\beta$ is a diffeomorphism, since so is $\beta$. 
The covering $p$ corresponds to the homomorphism $proj\pcirc\beta_*\: \pi_1(N\times S^1)\to \pi_1(S^1)$. 
The latter is equal to $proj\:\pi_1(N\times S^1)\to \pi_1(S^1)$ by the commutativity of~\eqref{piUnComm},
implying that $P\di N\times\bbr$.
\end{proof}

Closely related to \proref{P.invcob} is the following result.

\begin{Proposition}\label{P.invcob2}
Let $(W,j_M,j_N)$ be a cobordism between closed manifolds. 
The following five statements are equivalent:
\sk{1}
\begin{minipage}[t]{5.5cm}
\begin{itemize}
\item[(a)] $W$ is invertible.
\item[(b)] $W-j_N(N)\di M\times[0,\infty)$.
\item[(c)] $W-\partial W\di M\times \bbr$.
\end{itemize}
\end{minipage}
\begin{minipage}[t]{7cm}
\begin{itemize}
\item[]
\item[(b')] $W-j_M(M)\di N\times (-\infty,0]$.
\item[(c')] $W-\partial W\di N\times \bbr$.
\end{itemize}
\end{minipage}
\end{Proposition}

\begin{proof} It clearly suffices to prove this for a triad $(W,M,N)$.
We shall prove that $(a) \Rightarrow (b) \Rightarrow (c) \Rightarrow (a)$. 
The implication chain $(a) \Rightarrow (b') \Rightarrow (c') \Rightarrow (a)$
is obtained similarly.

Suppose that $W$ admits an inverse $W^\mun$.
Let $W_i$ and $W_i^\mun$ be copies of $W$ and $W^\mun$, indexed by $i\in\bbn$.
One has
$$
\begin{array}{rcl}
W-N &\di& W\pcirc N\times [0,\infty) \\  &\di&
W_0 \pcirc W_0^\mun \pcirc W_1 \pcirc W_1^\mun \pcirc W_2 \pcirc\cdots   \\  &\di&
M\times [0,1] \pcirc W_1 \pcirc W_1^\mun \pcirc W_2 \pcirc\cdots   \\  &\di&
M\times [0,\infty) \ ,
\end{array}
$$
thus $(a) \Rightarrow (b)$. 

As $(b) \Rightarrow (c)$ is obvious, it remains to prove $(c) \Rightarrow (a)$. 
For $1\leq r\in\bbn$, let $W_r = (M\times [-r,0])\pcirc W \pcirc (N \times [0,r])$ and $V_r=M\times [-r,r]$.
Let $f\:W-\partial W\to M\times\bbr$ be a diffeomorphism. As $W-\partial W\di\lim_{r\to\infty}W_r$
and $M\times\bbr\di\lim_{r\to\infty}V_r$, there are $1\leq r<s<t$ in $\bbn$ such that
$$
f(W_0) \subset V_r \subset f(W_s) \subset V_t  \, , 
$$
none of these inclusions being an equality.
As in the proof of \proref{P.invcob}, this provides classes $A,B,C,X,Y,Z$ in $\cob(M,M)$
such that $[V_r]=A\pcirc [W_0]\pcirc X$, $[W_s]=B\pcirc [V_r]\pcirc Y$ and $[V_t]=C\pcirc [W_s]\pcirc Z$.
Moreover, $B\pcirc A=[M\times [-s,0]]=\bbun_M$ and $C\pcirc B=[M\times [-t,-s]]=\bbun_M$. Therefore,
$B$ is invertible and $C=A=B^\mun$. In the same way, $Y$ is invertible and $X=Z=Y^\mun$.
Therefore, 
$$
[W]=[W_s] = B\pcirc [V_r]\pcirc Y = B\pcirc \bbun_M \pcirc Y = B\pcirc Y  
$$
and thus $W$ is invertible.
\end{proof}
\end{ccote}

\begin{ccote} The set $\calb(M)$. \ \rm
In view of \proref{P.invcob}, the study of the simplification problem is related to the classification of invertible 
cobordisms. We fix a smooth closed connected manifold $M$ and consider invertible cobordisms starting from $M$.
Two such cobordisms are regarded as equivalent if they are diffeomorphic relative to $M$. 
To be precise: $(W,j_M, j_N)$ is equivalent to $(W',j'_M, j'_N)$ if 
there is a diffeomorphism $f:W\di W'$ such that $j'_M=fj_M$.    
The equivalence class of a cobordism $(W,j_M,j_N)$ does not depend on $j_N$ and
is denoted by $[W,j_M[$, or just $[W[$.  Let $\calb(M)$ be the set of 
equivalence classes.

\begin{Example}\label{L.BMCf}\rm
Let $(W,j_M,j_N)$ be an invertible cobordism between closed manifolds $M$ and $N$. 
{\em  Then, $[W,j_M[=[M\times I,j_M^0[$ in $\calb(M)$ if and only if 
$[W]=[C_f]$ for some diffeomorphism $f\:M\to N$}.   
Indeed, the {\it if} part follows from \lemref{L.mapscob1}. Conversely, let $F\:M\times I\to W$ be a diffeomorphism $\rel M\times \{0\}$
and let $f\:M\to N$ be the restriction of $F$ to $M\times \{1\}$. 
Then $F^\mun\pcirc j_N=j_M^1\pcirc f^\mun$, which implies that $[W,j_M,j_N]=[M\times I, j_M^0,j_M^1\pcirc f^\mun]$.
The latter coincides with $[C_f]$ by \lemref{L.mapscob1} again.
\end{Example}

For any closed manifold $N$, the correspondence $[W]\mapsto [W[$ gives a map  
$\tilde\alpha_{M,N}\:\cob^*(M,N)\to \calb(M)$ that we shall now study 
(note that $\cob^*(M,N)$ is empty if $N$ is not invertibly cobordant to $M$).
The group $\aut(N)$ of self-diffeomorphisms of $N$ acts on the right on $\cob^*(M,N)$ by
$[W,j_M, j_N]\,\varphi = [W,j_M, j_N\pcirc\varphi]$. The map $\tilde\alpha_{M,N}$ is invariant
for this action and then descends to a map $\alpha_{M,N}\:\cob^*(M,N)/\aut(N)\to\calb(M)$.
We claim that the latter is injective. Indeed, if 
$\tilde\alpha_{M,N}([W,j_M, j_N])=\tilde\alpha_{M,N}([W',j_M', \hat j_N'])$, then 
there is a diffeomorphism $h\:W\to W$ such that $h\pcirc j_M=j_M'$ and thus
$$
[W,j_M, j_N]=[W',j'_M, h\pcirc j_N] = [W',j_M, j_N'] \, k\, ,
$$
where $k=(j_N')^\mun h\pcirc j_N \in \aut(N)$.

Let $\calm_n$ be the set of diffeomorphism classes of closed manifolds of dimension~$n$.
The correspondence $(W,M,N)\mapsto [N]$ defines a map 
\beq{D.mape}
e\:\calb(M)\to\calm_n  \, .
\eeq
Let $\calm_n^0$ be a set of representatives of $\calm_n$ (one manifold for each class).
\begin{Lemma}\label{L.partBM}
The map $\alpha=\dcup_{N\in\calm_n^0}\,\alpha_{\scr M,N}$ provides a bijection
$$
\dia{
\coprod_{N\in\calm_n^0} \cob^*(M,N)\bigg/\aut(N) \ar[r]^(0.72){\alpha}_(0.72){\approx} & \calb(M)  \, .
}
$$
The resulting partition of $\calb(M)$ is the one given by the preimages of the map~$e$.
\end{Lemma}

\begin{proof}
Let us first see that $\alpha$ is injective. 
Let $a\in\cob^*(M,N)$ and $b\in\cob^*(M,N')$ with $N,N'\in\calm_n^0$. If $\alpha(a)=\alpha(b)$, then 
$e\pcirc\alpha(a)=e\pcirc\alpha(b)$ and then $N=N'$, whence $a=b$ since $\alpha_{M,N}$ is injective.
To prove the surjectivity of $\alpha$, let $(W,j_M, j_N)$ be an invertible cobordism and
let $N_0$ be the representative of $e([W[)$ in $\calm_n^0$. 
Thus there exists a diffeomorphism $h\:N_0\to N$ and \\ 
$[W[=\alpha_{M,N_0}([W,j_M, j_N\pcirc h])$. 
\end{proof}

\begin{Remarks}\label{Bprime}\rm
\noindent (1) Composition of cobordisms defines an operation 
\beq{opcobb}
\cob^*(L,M)\times\calb(M)\fl{\pcirc}\calb(L) \, ,
\eeq
making $\calb$ a functor on the category of closed manifolds and 
(equivalence classes of) invertible cobordisms.\sk{2}

\noindent (2) There is a version $\calb'(M)$ of $\calb(M)$ where we only use
{\em triples} $(W, M, N)$. The obvious inclusion $\calb'(M)\to \calb(M)$ is, in 
fact, a bijection, by the observation at the end of \ref{Co.cobordism}. This will
often be usedwithout further mention, to simplify notation.

Note that, using \lemref{L.mapscob1}, the map $\calb(M)\to \calb'(M)$ can
 also be defined as $[W,j_M, j_N]\mapsto [C_{j_M}\pcirc(W,\id_{M'},j_N)]$, 
where $M'=j_M(M)$.
\end{Remarks}
\end{ccote}

\begin{ccote}\label{Co.hcob} h-cobordisms.  \rm \
A cobordism $(W,j_M,j_N)$ from $M$ to $N$ is called an \dfn{h-cobordism} if both of the 
maps $ j_M:M\to W$ and $ j_N:N\to W$ are homotopy equivalences.
The composition of $j_N$ with a homotopy inverse of $j_M$ then produces a homotopy equivalence
$h:N\to M$ whose homotopy class is well defined. Any choice of such an $h$ will be called a
\dfn{natural homotopy equivalence associated to $W$}.
The main relationship between h-cobordisms and invertible cobordism is given by the following proposition.

\begin{Proposition}\label{P.ihcob}
An invertible cobordism is an h-cobordism. The converse is true when $n\ne 3$. 
\end{Proposition}

The above statement is unknown for $n=3$.

\begin{proof} It suffices to consider the case of an invertible triad
$(W,M,N)$.  Let  $(W',N,M)$ be an inverse for $W$, and choose  diffeomorphisms
 $W\pcirc W' \stackrel{\approx}\to M\times I$ rel $M$ and  
$W'\pcirc W \stackrel{\approx}\to N\times I$ rel $N$. 
The inclusions
$M\subset W \subset W \pcirc W'$ and 
$W\subset W \pcirc W' \subset W \pcirc (W' \pcirc W)\approx W$ show that
$M$ and $W$ are homotopy retracts of each other. Analogously for $N$ and $W$.
\smallskip

That an h-cobordism is invertible when $n\geq 5$ will be proven in \thref{T.hcob}.
For $n=4$, this is a result of Stallings (see \cite[Thm~4]{Stallings}), and for $n\leq2$ it follows 
from (the proof of) \proref{P.BSurf}.
\end{proof}
\end{ccote}

\begin{ccote}\label{Co.tau} Whitehead torsion. \ \rm We recall here some facts about Whitehead torsion and the s-cobordism theorem.
For more details, see \cite{MilnorWT,CohenBook}. 

The Whitehead group $\wh(\pi)$ of a group $\pi$ is defined as
\beq{defWh}
\wh(\pi) = GL_\infty(\bbz\pi)\Big/ E_\infty(\bbz\pi)\cup (\pm\pi) \, ,
\eeq
where $E_\infty(\bbz\pi)$ is the subgroup of elementary matrices and $(\pm \pi)$ denote the subgroup 
of $(1\times \{1\})$-invertible matrix $(\pm \gamma)$ with $\gamma\in\pi$. As $E_\infty(\bbz\pi)$
is the commutator of $GL_\infty(\bbz\pi)$, the group $\wh(\pi)$ is abelian.

A pair $(X,Y)$ of finite connected CW-complexes is an \dfn{h-pair} if the inclusion $Y\hookrightarrow X$ is a homotopy equivalence.
To such a pair is associated its \dfn{Whitehead torsion} $\tau(X,Y)\in\wh(\pi_1Y)$. The Whitehead torsion $\tau(f)\in \wh(K)$ 
of a homotopy equivalence $f\:K\to L$ ($K$, $L$ finite CW-complexes) is defined by $\tau(f)=\tau(C_f,K)$, 
where $C_f$ is the mapping cylinder of $f$. If $\tau(f)=0$, we say that $f$ is a \dfn{simple homotopy equivalence}.

If $K\fl{f} L \fl{g} M$ are homotopy equivalences between finite CW-complexes, then
\beq{comtors1}
\tau(g\pcirc f) = \tau(f) + (f_*)^\mun(\tau(g))
\eeq
where $f_*\:\wh(\pi_1L)\to\wh(\pi_1K)$ is the isomorphism induced by $f$.  
Also useful is the following partial product formula. 
Let $K$, $L$ and $Z$ be connected finite CW-complexes and let $f\:K\to L$ be a homotopy equivalence.
Then, in $\wh(\pi_1(K\times Z))$, one has
\beq{parProdFor}
\tau(f\times \id_Z) = \chi(Z)\cdot\tau(f) \, ,
\eeq
where $\chi(Z)$ is the Euler characteristic of $Z$ (see \cite[(23.2)]{CohenBook}).

\begin{Remark}\rm This definition of the torsion of a homotopy equivalence is 
slightly non-standard, as it measures the torsion in the Whitehead group of
the source of $f$, rather than the target, as in \cite{CohenBook} and 
\cite{MilnorWT}.  The two definitions are of course equivalent, but for our 
purposes, the current definition is more convenient,
since now the torsion of a pair $(X,Y)$ is equal to the torsion of the 
inclusion map $X\subset Y$.
\end{Remark}

An easy case for computing $\tau(X,Y)$ is when the h-pair $(X,Y)$ is in \dfn{simplified form}, i.e.
\beq{simform}
X = Y \,\cup\, \bigcup_{i=1}^p\, e_i^r \,\cup\, \bigcup_{i=1^p}\, e_i^{r+1}  \quad (r\geq 2) \,
\eeq 
where $e_i^j$ denotes a $j$-cell. 
Let $(\tilde X,\tilde Y)$ be the pair of universal covers. Then the chain complex of 
$C_*(\tilde X,\tilde Y)$ is a complex of free $\bbz\pi$-modules and the boundary operator 
$\delta\: C_{r+1}(\tilde X,\tilde Y)\to C_{r}(\tilde X,\tilde Y)$ is an isomorphism. 
Bases may be obtained for $C_*(\tilde X,\tilde Y)$ by choosing orientations of $e_i^j$
and liftings $\tilde e_i^j$ in $\tilde X$. 
Then,for such bases, $\tau(X,Y)$ is represented in $GL_p(\bbz\pi)$ 
by the matrix of $\delta^\varepsilon$ with $\varepsilon=(-1)^{(r-1)}$.

Let $M$ be a connected manifold. The Whitehead group $\wh(\pi_1M)$ is then endowed with an involution 
\beq{defInvol}
\tau\mapsto\bar\tau
\eeq
induced by the anti-automorphism of 
$\bbz\pi_1M$ satisfying $\bar a = \omega(a)a^\mun$ for $a\in\pi_1M$, 
where $\omega\:\pi_1M\to \{\pm 1\}$ 
is the orientation character of $M$. We denote by $\wh(M)$ the abelian group $\wh(\pi_1M)$
equipped with this involution.

Let $W$ be an invertible cobordism starting from the closed connected manifold $M$. 
Then $(W,M)$ admits a $\calc^1$-triangulation which is unique up to PL-homeomorphism \cite[Theorems~7 and~8]{WheadTri}.
This makes $(W,M)$ an h-pair whose Whitehead torsion $\tau(W,M)\in\wh(M)$ is well defined.
An invertible cobordism with vanishing torsion is called an \dfn{s-cobordism}.

To compute $\tau(W,M)$, one can use a simplified form analogous to~\eqref{simform}. 

\begin{Lemma}\label{L.simformCob}
Let $(W,M,N)$ be an invertible cobordism with $\dim M =n\geq 4$. Then, for $2\leq r\leq n-2$,
there exists a decomposition 
$$
W = W_r \pcirc W_{r+1}
$$
where $(W_r,M,M_r)$ has a handle decomposition starting form $M$ with
only handles of index $r$
and  
$(W_{r+1},M_r,N)$ has a handle decomposition starting form $M_r$ with 
only handles of index $r+1$.
\end{Lemma}

\begin{proof}
When $n\geq 5$, this is \cite[Lemma~1]{KeBMS}. We have to see that the proof works for $n=4$. 
The principle is to eliminate  handles of index $k$ by 
replacing them by  handles of index $k+2$. There is an easy argument eliminating $0$-handles,
which also works when $n=4$. There is also a special argument to get rid of $1$-handles, given in 
\cite[pp.~35--36]{KeBMS}. This argument also works when $n=4$: it suffices to prove that 
two embeddings $f_0,f_1$ of $S^1$ into a 4-dimensional manifold $P$ which are related by a homotopy $f_t$ are 
ambient isotopic. Let $f\:S^1\times I \to P\times I$ be the map $f(x,t)=(f_t(x),t))$. By general position,
$f$ is homotopic relative $S^1\times\partial I$ to an embedding. Therefore, $f_0$ and $f_1$ are concordant and,
as we are in codimension 3, they are ambient-isotopic \cite{Hudson}.
\end{proof}

The number of handles for $W_{r+1}$ and $W_r$ is the same (say, $p$) since $M\hookrightarrow W$ is a homotopy equivalence.
As a consequence (see \cite[p.~83]{RouSan}), $(W,M)$ retracts by deformation relative $M$ onto a CW-pair $(X,M)$ 
as in~\eqref{simform} from which we can compute $\tau(W,M)=\tau(X,M)$.

Torsions of invertible cobordisms satisfy some specific formulae. 
First, let $(W,M,N)$ and $(W',N,N')$ be invertible cobordisms. Then, in $\wh(M)$, one has
\beq{comphcob}
\tau(W\pcirc W',M) = \tau(W,M) + h_*(\tau(W',N))  \, ,
\eeq
where $h_*\:N\to M$ is a natural homotopy equivalence associated to $W$.
This follows from \cite[(20.2) and (20.3)]{CohenBook}.  
One also has the \dfn{duality formula} (see \cite[pp.~394--398]{MilnorWT}):
\beq{duality}
h_*(\tau(W,N)) =  (-1)^n\overline{\tau(W,M)} \, .
\eeq

More generally, if $(W,j_M)$ represents an element in $\calb(M)$, we define
$$
\calt(W,j_M)=\tau(j_M)=({j_M}_*)^{-1}\tau(W,j_M(M)).
$$

The duality formula now becomes 
\beq{duality2}
( j_N)_*^{-1} ( j_M)_*(\tau(W,j_N)) =  (-1)^n\overline{\tau(W, j_M)} \, .
\eeq
Thanks to the uniqueness of $\calc^1$-triangulations, this gives a well 
defined map
$$
\calt\: \calb(M) \to \wh(M).  
$$

\begin{Theorem}\label{T.hcob}
Let $M$ be a smooth closed connected manifold of dimension $\geq 5$. Then,
\begin{itemize}
\item[(i)] the map $\calt\: \calb(M) \to \wh(M)$ is a bijection;
\item[(ii)] any h-cobordism $(W,M,N)$ is invertible; 
\item[(iii)] $\calt(W,j_M)=0$ if and only if $(W,j_M(M))\di (M\times I,M\times \{0\}) \ (\rel M)$. 
\end{itemize}
\end{Theorem}

For the situation when $n=3,4$, see \lemref{L.surj4}, the end of \secref{S.nfour} and \secref{S.nthree}.

\begin{proof} The proof involves four steps.
 
\sk{1}\noindent
(1) \em Part (iii). \ \rm
This is the content of the {\it s-cobordism theorem}, which is valid for $n\geq 5$.
This theorem was first independently proved by Barden, Mazur and Stallings in the early 60's. 
For a proof and references, see \cite{KeBMS}.

\sk{1}\noindent (2) \em For any $\tau\in\wh(M)$, there exists an h-cobordism $(V,M,M'))$ with 
$\tau(V,M)=\tau$.  \ \rm This was proven in \cite[Theorem~11.1]{MilnorWT}).

\sk{1}\noindent (3) \em Part (ii). \ \rm Let $(W,M,N)$ be an h-cobordism and let $\sigma=\tau(W,M)$.
Let $f\:N\to M$ be the composition of the inclusion $N\hookrightarrow W$ with a 
retraction from $W$ to $M$. Let $(W_R,N,M_R)$ be an h-cobordism such that $f_*(\tau(W_R,N))=-\sigma$.
By~\eqref{comphcob}, one has
$$
\tau(W\pcirc W_R,M) = \tau(W,M) + f_*(\tau(W_R,N)) = 0 \, .
$$
By part~(iii) already established, there exists a diffeomorphism (relative $M$)  $H\:W\pcirc W_R\to M\times I$. 
Let $h\:M_R\to M\times \{1\}$ be the restriction of $H$ to $M_R$. Using the diffeomorphism $H$ and \lemref{L.mapscob1},
one gets
$$
[W\pcirc W_R,j_M\dcup j_{M_R}] = [M\times I,j^0_M\dcup h] = [C_h] \, .
$$
Therefore, $[W]\pcirc [\hat W_R]=\bbun_M$, where $[\hat W_R] = [W_R]\pcirc [C_{h^\mun}]$.

Similarly, let $(W_L,M_L,M)$ be an h-cobordism with $\tau(W_L,M)=(-1)^{n+1}\bar\sigma$. 
By\eqref{duality} and~\eqref{comphcob}, one has 
\beq{T.hcob-eq10}
f_*(\tau(W_L\pcirc W,N)) = f_*\big(\tau(W,N) + f_*^\mun(\tau(W_L,M))\big) = (-1)^n\bar\sigma +(-1)^{n+1}\bar\sigma 
=0 \, .
\eeq
As above, this permits us to construct a cobordism $\hat W_L$ from $N$ to $M$, 
such that 
$\hat W_L$ 
is a left inverse for $W$: $[\hat W_L]\pcirc [W] = \bbun_N$. 
Having a left and right inverse, $[W]$ is invertible and $[W_L]=[W_R]$.  

\sk{1}\noindent (4) \em Part (i). \ \rm The surjectivity of $\calt$ follows from (2) and part~(ii) already proven.
For the injectivity, let $(W,M,N)$ and $(W',M,N')$ be two invertible cobordisms starting from $M$, 
with $\tau(W,M)=\tau(W',M) = \alpha$. As $\calt$ is surjective, there is an invertible cobordism $(V,P,M)$
such that $\tau(V,M)=(-1)^{n+1}\bar\alpha$.
As in~\eqref{T.hcob-eq10}, we check that $\tau(V\pcirc W,P)=\tau(V\pcirc W',P)=0$. By (1) above, 
there are diffeomorphisms (relative $P$) $H\:V\pcirc W\to P\times I$ and $H'\: V\pcirc W'\to P\times I$,
with restrictions $h\:N\to P\times \{1\}\approx P$ and $h'\:N'\to P\times \{1\}\approx P$. Then
$$
[V]\pcirc [W[ \, = [P\times I[ \, = [V]\pcirc [W'[  \ .  
$$
As $[V]$ is invertible, one gets the equality $[W[=[W'[\,$ in  $\calb(M)$.
\end{proof}

\begin{Remark}\label{MxS1usingscob} \rm
The results of this section may be used to give an alternative proof
that {\em two closed manifolds $M$ and $N$  of dimension $\geq 4$ 
which are h-cobordant are $\bbr$-diffeomorphic} (\proref{P.invcob}). 
Indeed, let $(W,N,M)$ be an h-cobordism.
Then, $W\times S^1$ is an s-cobordism by \eqref{parProdFor} and thus, using 
\thref{T.hcob}, there exists a diffeomorphism $F\:N\times S^1\times I\to W\times S^1$ 
inducing a diffeomorphism $F^1\:N\times S^1\times \{1\}\to M\times S^1$. 
By \proref{P.invcob}, one deduces that $M\sdi N$. Indeed, 
Condition (c) of \proref{P.invcob} may be checked for $\beta=F_1$, using that $F$ may be chosen relative 
$N\times S^1\times \{0\}$.  
\end{Remark}
\end{ccote}

\begin{ccote}\label{concordance}  Remarks on the relative case. 
Concordance. \rm\ 

With minor modifications most of the results in this section go through also
in the relative case, i.\,e. when $M$ and $N$ have nonempty boundaries.  
In particular, we can define invertible cobordisms and relative invertible cobordisms
the same way in this generality. Moreover, the crucial results used in this 
section, the s-cobordism theorem and classification of h-cobordisms by
Whitehead torsion still hold.  Although they are usually only formulated in 
the closed case, the proofs don't really use this, but work exactly the same
way in general, since all the constructions can be done `away from the 
boundary'.  This means that \thref{T.hcob} could just as well have been 
formulated for manifolds with boundary, to the expense of a little more 
notation.

Here we will not need a full discussion of this, but in  \secref{S.classif} we 
come back to a special case, when we wish to
compare invertible cobordisms between the same manifold, using the relation of 
{\em concordance}.

Fix two invertibly cobordant closed manifolds $M$ and $N$, and let
$(W,j_M,j_N)$ and $(W',j'_M,j'_N)$ be two invertible cobordisms between them. 
We say that these cobordisms are {\em concordant} if there is an invertible
cobordism $(X,J_W,J_{W'})$ between them, with the following extra compatibility
condition between $J_*$'s and $j_*$'s: There are embeddings 
$H_M:M\times I\to \partial X$ and $H_N:N\times I\to \partial X$ filling in
$\partial X-(J_WW\cup J_{W'}W')$ and such that $J_W j_M=H_M j^0_M$, 
$J_{W'} j_M'=H_M j^1_M$, $J_W j_N=H_M j^0_N$ and  $J_{W'} j_N'=H_N j^1_N$.

Observe that concordance defines an equivalence relation on  
$Cob^*(M,N)$.  We denote
the set of equivalence classes by $\overline{Cob^*}(M,N)$. Via the
composed map $Cob^*(M,N)\to \calb(M)\to \wh(M)$ this relation corresponds to a relation
on $\wh(M)$, which will be important in \secref{S.classif}. 

\begin{Lemma}\label{L.WhConc} Let $M$ and  $N$ be a compact closed 
manifolds of dimension $n$,
 let $(W,j_M,j_N)$ and $(W',j'_M,j'_N)$ be two invertible cobordisms, and
assume $(X,J_W,J_{W'})$ is a concordance between them. The Whitehead torsions
are then  related by the formula
$$
\tau(W',j'_M)-\tau(W,j_M)={ j_M}{}_*^{-1}(\tau(X,J_W)+(-1)^n\overline{\tau(X,J_W)}\,).
$$
\end{Lemma}
\begin{proof} The two maps $ j_W j_M$ and $ j_{W'} j'_M$ are
homotopic homotopy equivalences. Hence they have the same torsion, and we get 
the identity
$$
\tau( j_M)+{ j_M}{}_*^{-1}(\tau( j_W))=
\tau( j'_M)+{ j'_M}{}_*^{-1}(\tau( j_{W'})).
$$
The result now follows from the duality formula
(\eqref{duality2}).
\end{proof}
\end{ccote}

\section{The case $n\geq 5$}\label{S.nfive}
The following theorem is a direct consequence of \proref{P.invcob} and \thref{T.hcob}.

\begin{Theorem}\label{T.n5}
Let $M$ and $N$ be smooth closed connected manifolds of dimension $n\geq 5$ such that 
$N\sdi M$. Suppose that $\wh(M)=0$. Then $N\di M$.
\end{Theorem}

As $\wh(\{1\})=0$, \thref{T.n5} implies Theorem~A in the case $n\geq 5$. As a first generalization, 
let us consider the following conjecture.

\begin{Conjecture}\label{CFJ}
Let $M$ and $N$ be smooth connected closed manifolds of dimension $\geq 5$ such that $N\sdi M$.
Suppose that $\pi_1M$ is torsion-free. Then $N\di M$. 
\end{Conjecture}

Using \thref{T.n5}, Conjecture~\ref{CFJ} would follow from the well known conjecture
that $\wh(\pi)=0$ if $\pi$ is a torsion-free finitely presented group.
This is part of the Farrell-Jones conjecture in K-theory and it has been proven by several authors for various classes
of finitely presented torsion-free groups, such as 
{\em free abelian groups, free groups, virtually solvable groups, 
word-hyperbolic groups, CAT(0)-groups, etc}. For references, see 
\cite{LueckReich,BaLR} (see also the proof of \thref{T.n4}).

To generalize \thref{T.n5} we need to introduce the concept of \dfn{inertial} 
invertible cobordisms:
a cobordism $(W,j_M,j_N)$ is \dfn{inertial} if $N\di M$.

Let $\cali\calb(M)$ be the subset of elements in $\calb(M)$
represented by inertial cobordisms and let $I(M)=\calt(\cali\calb(M))\subset\wh(M)$.
Note that $I(M)$ is not a subgroup of $\wh(M)$ in general \cite[Remark~6.2]{Hau2}.

\thref{T.hcob} together with \proref{P.invcob} implies 
the following result, which is the strongest possible generalization of \thref{T.n5}:

\begin{Theorem}\label{T.n5gen}
For $M$ a smooth connected closed manifold of dimension $\geq 5$, the following assertions are equivalent.
\begin{itemize}
\item[(i)] Any manifold $\bbr$-diffeomorphic to $M$ is diffeomorphic to $M$.
\item[(ii)] $I(M)=\wh(M)$.   \mancqfd
\end{itemize}
\end{Theorem}

The set $I(M)$ is contained in the set $I_{\rm TOP}(M)$ of those 
$\sigma\in\wh(M)$ such that if $(W,M,N)$
is an invertible cobordism with $\tau(W,M)=\sigma$, then $N\ho M$. In all
cases where these sets are computed, they are equal, but it is 
not known whether $I(M)=I_{\rm TOP}(M)$ in general for a smooth manifold $M$ of dimension $\geq 5$, contrary to the claim in \cite{JaKwa}.  
However, there is a smaller set, $SI(M)$, of {\em strongly} inertial invertible cobordisms, which indeed is the 
same in the two categories.  This is the set of invertible cobordisms  $(W,j_M, j_N)$ such
that $j_M^{-1}\pcirc j_N$ is homotopic to a diffeomorphism (homeomorphism).
See \cite{JaKwaIN}. \par

The general question is intriguing, not the least because of the following
reformulation:\smallskip

\begin{Question}
Given two smooth manifolds $M$ and $N$ of dimension $\neq 4$ such that
$M\sdi N$ and $M\ho N$.  Is $M\di N$?
\end{Question}

The answer of the above question is ``infinitely no'' in dimension 4, even if $M$ and $N$ are simply connected
(see \secref{S.nfour}). It is ``yes'' in dimension 3 for orientable manifolds (see Theorem~C).

\begin{Examples}\label{IneqWh} \rm 
We start with examples where $I(M)\neq\wh(M)$. 

\sk{1}
(1) $I_{\rm TOP}(M)\neq \wh(M)$ for $M=L(7,1)\times S^4$ or $M=L(7,2)\times S^4$. Indeed, in 
1961, J.~Milnor \cite{MilnorTwoCom} showed that these two manifolds 
are invertibly cobordant but have not the same simple homotopy type 
(they are then not homeomorphic by Chapman's theorem \cite[Appendix]{CohenBook}).
Historically, this was the first example of this kind and
Milnor used it to produce the first counterexample to the Hauptvermutung 
for finite simplicial complexes \cite{MilnorTwoCom}.
 
(2) $I_{\rm TOP}(M)=0$ if $M$ is a lens space of dimension $\geq 5$  \cite[Corollary~12.13]{MilnorWT}.
This result was extended in \cite{KwasikSchultz} to generalized spherical spaceforms (see~\ref{spaceforms}).

(3) For $k\geq 3$, one has $I_{\rm TOP}(L(p,q)\times S^{2k})=0$ if $p\equiv 3 \ ({\rm mod} 4)$. Also, 
$I(L(5,1)\times S^{2k})=0$ but there exists a manifold $N$ $h$-cobordant to $L(5,1)\times S^{2k}$
such that $I(N)\neq 0$ (see \cite[\S\,6]{Hau2}). 

(4) Let $W$ be an invertible cobordism and consider its dual $\bar W$ (see \ref{Co.cobordism}).  
Then, $W\pcirc\bar W$ is an inertial invertible cobordism. 
By~\eqref{comphcob} and~\eqref{duality}, one has $\tau(W\cup \bar W,M)=\tau(W,M) + (-1)^n\overline{\tau(W,M)}$.
Therefore $\caln(M)=\{\tau+ (-1)^{n}\bar\tau\mid \tau\in\wh(M)\}\subset I(M)$. The subgroup $\caln(M)$
plays an important role in \secref{S.classif}. 

(5) Let $\pi$ be a finite group such that $\wh(\pi)$ is infinite. 
(For $\pi$ abelian, this is the case unless $\pi$ has exponent $2,3,4$ or $6$: see \cite{Bass}).
Then, in every odd dimension $\geq 5$, there are manifolds $M$ with fundamental group $\pi$
such that $I_{\rm TOP}(N)$ is finite for any manifold $N$ invertibly cobordant to $M$ (see \cite[Theorem~1.2 and its proof]{JaKwa}). Then there are infinitely 
many distinct homeomorphism classes of manifolds $\bbr$-diffeomorphic to $M$.
\end{Examples}

In view of \thref{T.n5gen}, the case $I(M)=\wh(M)$ is particularly interesting. 
The proof of the following proposition uses a standard technique to produce $h$-cobordisms,
going back to \cite[\S\,2]{MilnorTwoCom} and generalized  independently in \cite{Lawson1} and \cite{Hau1}.  

\begin{Proposition}
Let $K$ be a finite $2$-dimensional polyhedron with $\pi_1K$ finite abelian and let $n\geq 5$.
Let $E$ be a regular neighborhood of an embedding of $K$ in $\bbr^{n+1}$ and let $M=\partial E$.
Then $I(M) = \wh(M)$.
\end{Proposition}

\begin{proof}
Let $i\:K\to E$ be the natural inclusion and let $f\:K\to K$ be a homotopy equivalence 
with homotopy inverse $\varphi$. 
Then, $i\pcirc f$ is homotopic to an embedding $j_f\:K\to E$. Let $V_f$ be a regular neighborhood of $j_f(K)$
in $E$ and let $W_f=E-{\rm int\,}V_f$. Doing the same construction in $V_f$ with $j_f\pcirc\varphi$,
and another time using again $f$, shows that $(W_f,\partial V_f,M)$ is an invertible cobordism.
The torsion of $W_f$ is related to $\tau(f)$, via natural
identifications of fundamental groups (see \cite[proof of Proposition~1.1]{Hau1} or \cite[Proposition~3]{Lawson1}). 
As $j_f$ is isotopic to $i$ in $\bbr^{n+1}$, one has $E\di V_f$, thus $W_f$ is inertial.
By \cite[Theorem 1]{Latiolais}, every element of $\wh(\pi_1K)$ is realizable as the torsion of
a self homotopy equivalence of $K$. This proves that $I(M) = \wh(M)$.
\end{proof}

In the even case, this result has a vast generalization, as a consequence of the following proposition.

\begin{Proposition}\label{P.sigbar}
Let $M$ be a smooth connected closed manifold of dimension $n\geq 5$. Let $\sigma\in\wh(M)$ such that $\sigma=(-1)^n\bar\sigma$. 
Then $\sigma\in I(M)$.  
\end{Proposition}

\begin{proof}
Let $i\:K\to M$ be an embedding of a finite connected 2-dimensional complex $K$ into $M$ such that $\pi_1i\:\pi_1(K)\to \pi_1(M)$
is an isomorphism, which we use to measure Whitehead torsions in $\pi_1(K)$. 
Let $A$ be a regular neighborhood of $i(K)$ and let $B=M-{\rm int}A$. 

Let $(V,A,A')$ be an invertible cobordism relative boundary with $\tau(V,A)=\sigma$. Then, $W=V\cup(B\times I)$ is an invertible
cobordism from $M$ to $M'=A'\cup(B\times \{1\})$ with $\tau(W,M)=\sigma$. \par

Since $\dim M\geq 5$ and $\mathrm{codim\ } K\geq 3$,  we have 
$\dim \partial A\geq 4$ and 
$\pi_1\partial A =\pi_1 A$.  Then, by Theorem 3.11 and Lemma 5.6, there 
also exists  an
invertible cobordism $T\in \calb{(\partial A)}$ with Whitehead torsion $\sigma$.
The condition $\sigma=(-1)^n\bar\sigma$ now means that $T^{-1}=\bar T$, and
$A\pcirc T\pcirc\bar T\di A$, rel $\partial$. 

Let $C=A\pcirc T$. Then we may also consider $V$ as an h-cobordism from $C$ to
$A'\pcirc T$, and computing the torsion of the inclusion $K\subset V$
two ways, we see that $\tau(V,C)=0$.  By the s-cobordism theorem 
we conclude that
$C\di A'\pcirc T$ rel $\partial$, and hence $A'\di A$ rel $\partial$, 
since $T$ is invertible. Extending
this diffeomorphism by the identity on $B$, we see that $M'\di M$. 
\end{proof}

\begin{Remark} \rm When $\sigma\neq (-1)^n\bar\sigma$, it is still possible that $M'\di M$,
as seen above; simply, the diffeomorphism from $M'$ to $M$ is not relative $B$.
\end{Remark}

When $M$ is orientable with $\pi_1M$ finite abelian, then $\bar\sigma=\sigma$ for all $\sigma\in\wh(M)$ \cite{Bak},
hence we have the following corollary of \proref{P.sigbar}.

\begin{Corollary}
Let $M$ be a connected orientable closed manifold of even dimension $\geq 6$
such that $\pi_1M$ finite abelian. Then $I(M)=\wh(M)$. \mancqfd
\end{Corollary}

In the case when $\pi_1(M)$ is finite cyclic, this was first proved in
 \cite[Cor. 1]{Lawson1}\sk{2}

We also mention another corollary of \proref{P.sigbar}, which essentially 
amounts to a curious reformulation. Let $(W,M,N)$ be an invertible cobordism with 
Whitehead torsion $\sigma = \tau(W,M)$, and 
let $h:N\to M$ be a natural homotopy equivalence associated to $W$. It follows 
easily from the composition and duality formulae (\ref{comtors1}) and
(\ref{duality}) that  
$\tau(h) = - \sigma +(-1)^n\bar\sigma$.
Hence we see that $h$ is a simple homotopy equivalence if and only if
$ \sigma = (-1)^n\bar\sigma$.

\begin{Corollary}
If the natural homotopy equivalence defined by the invertible cobordism $(W,M,N)$ is
simple, then $(W,M,N)$ is inertial.
\end{Corollary}

But note that $h$ may not itself be homotopic to a homeomorphism! 
A counterexample is given in \cite[Example 6.4]{JaKwa}.\medskip

Finally, we describe how to get inertial invertible cobordisms by 
``stabilization'' (up to connected sums with $S^r\times S^{n-r}$).
First, a few words about connected sums. 
Since we do not worry about orientations, the diffeomorphism type $M_1\,\sharp\,M_2$ may depend on the choice
of embeddings $\beta_i\:D^n\to M_i$ (see e.g. \cite[\S\,4.2.3]{HauBook}). This will not bother us because
our manifold $M_2$ (like $S^r\times S^{n-r}$) 
admits an orientation reversing diffeomorphism. The same holds true for \dfn{cobordism connected sum}
$W_1\,\sharp\,W_2$, obtained using embeddings $\beta_i\:(D^n\times I,D^n\times \{0\},D^n\times \{1\})
\to (W_i,M_i,N_i)$. 

\begin{Proposition}{{\rm (\cite{Hatcher-Lawson}, compare~\ref{cstab})}}\label{P.HL}
Let $M$ be a smooth connected closed manifold of dimension $n\geq 5$. Let $(W,M,N))$ be an invertible cobordism
such that $\tau(W,M)$ is represented by a matrix in $GL_p(\bbz\pi_1M)$. Then, for $2\leq r\leq n+2$, 
$$
M\,\sharp\, p(S^r\times S^{n-r}) \di N\,\sharp\, p(S^r\times S^{n-r}) \, .
$$
Consequently, the cobordism $W\,\sharp\, p(S^r\times  S^{n-r}\times I)$ is an inertial invertible cobordism. 
\end{Proposition}

\begin{proof} 
One uses a simplified handle decomposition $W = W_r \pcirc W_{r+1}$ like in 
\lemref{L.simformCob}, together with the remark of \cite{Hatcher-Lawson} 
that the $r$-handles of $(W_r,M,M_r)$ are attached trivially, meaning that the attaching embedding factors through the standard embedding of 
$S^{r-1}\times D^{n+1-r}$ into $\bbr^n$.
This implies that $M_r\di M\,\sharp\, p(S^r\times S^{n-r})$. The same holds true for
the $(n-r)$-handles of $(\bar W_{r+1},N,M_r)$, thus 
$M\,\sharp\, p(S^{r}\times S^{n-r})$. For details, see~\cite{Hatcher-Lawson}. 
\end{proof}

Combined with \proref{P.invcob}, this gives an interesting 
relation between two kinds of stabilization: 

\begin{Corollary}\label{sum.stab}
Let $M$ and $N$ be closed smooth manifolds of dimensions $\geq 5$
which are $\bbr$-diffeomorphic. 
Then  there exists an integer $p$ such that
$M \,\sharp\, p(S^{r}\times S^{n-r})\di N\,\sharp\, p(S^{r}\times S^{n-r})$
for any $r$ such that $2\leq r\leq n-2$.  
If $\pi_1(M)$ is finite, $p$ may be chosen to be less than or equal to 2. 
\end{Corollary}

\begin{proof} The last statement follows since  
$GL_2(\bbz\pi)\to\wh(\pi)$ is surjective if $\pi$ is a finite group \cite{Vas}.
Note that $p$ can not always be chosen to be 1 (see \cite[Theorem~1.1]{JaKwa}).
\end{proof}

 An intriguing
question is if there is some kind of converse to this result. 
  A very special case is given by Lemma 4.1 in \cite{JaKwa}\sk{2}
 
By \thref{T.hcob}, an invertible cobordism $X$ starting from $Y=M\,\sharp\, p(S^r\times S^{n-r})$
is of the form $W\,\sharp\, p(S^r\times  S^{n-r}\times I)$ where $W$ is an invertible cobordism
starting from $M$ with $\tau(X,Y)=\tau(W,M)$. Using \proref{P.HL}, this proves the following

\begin{Corollary}
Let $M$ be a smooth connected closed manifold of dimension $n\geq 5$.
Suppose that $GL_p(\bbz\pi_1M) \to \wh(\pi_1M)$ is surjective.
Then, for any $2\leq r\leq n+2$, one has $I(M\,\sharp\, p(S^r\times S^{n-r}))=\wh(M)$. 
\mancqfd
\end{Corollary}

\section{The case $n=4$}\label{S.nfour}

A group $\pi$ is called \dfn{poly-(finite or cyclic)} if it admits an ascending sequence of 
subgroups, each normal in the next, with successive quotients either finite or cyclic
(this is equivalent to $\pi$ being virtually polycyclic: see \cite[Theorem~2.6]{Wehrfritz}).
We first prove the following theorem which implies part~(ii) of Theorem~A.

\begin{Theorem}\label{T.n4}
Let $M$ and $N$ be smooth connected closed manifolds of dimension $4$ such that 
$N\sdi M$. Suppose that $\pi_1M$ is poly-(finite or cyclic) and that $\wh(M)=0$. Then $N\ho M$. 
\end{Theorem}

\begin{proof}
By \proref{P.invcob}, there is an invertible cobordism $W$ from $M$ to~$N$. 
Then $W$ is an h-cobordism by \proref{P.ihcob} and, as $\wh(M)=0$, it is an s-cobordism. 
The topological s-cobordism theorem in dimension $4$ holds for closed manifold
with poly-(finite or cyclic) fundamental group \cite[Theorem~7.1A and the Embedding theorem p.~5]{Freedman-Quinn-Book}.
Therefore, $W\ho M\times I \ (\rel M)$ and then $N\ho M$. 
\end{proof}

\begin{Example}\rm 
By \cite{Farrell-Hsiang}, $\wh(M)=0$ when $\pi_1M$ is poly-(finite or cyclic) and torsion-free. 
By \thref{T.n4}, $N\sdi M$ implies $N\ho M$ in this case. 
\end{Example}

\begin{Remark}\rm
Poly-(finite or cyclic) groups are 
the only known examples of finitely presented groups which are called ``good'' by Freedman and Quinn, i.e.
for which their techniques work \cite[p.~99]{Freedman-Quinn-Book}. Freedman and Teichner \cite{FreedTeich} showed that groups 
of subexponential growth are good, but the only known such groups which are finitely presented are 
poly-(finite or cyclic). Note that \thref{T.n4} may be true even if $\pi_1(M)$ 
is not good in the above sense.
\end{Remark}

We now prepare the proof of Theorem~B of the introduction.
Recall that, to a homeomorphism $f\:M\to N$ between manifolds is associated its 
Casson-Sullivan invariant ${\rm cs}(f)\in H^3(M;\bbz_2)$,
which, for $\dim M \geq 4$, vanishes if and only if $h$ is isotopic to a PL-homeomorphism 
(thus, to a diffeomorphism if $\dim M = 4$: see \cite[Definition~3.4.5]{Rudyak}).

\begin{Proposition}\label{P.CS}
Let $M$, $N$ be two closed smooth connected $4$-manifolds. Suppose that there exists a homeomorphism $f\:M\to N$
with vanishing Casson-Sullivan invariant. Then, $M$ and $N$ are smoothly s-cobordant. The converse is true when 
$\pi_1(M)$ is poly-(finite or cyclic).
\end{Proposition}
 
\begin{proof}
The mapping cylinder $C_f$ produces a topological s-cobordism $W$ between $M$ and $N$. 
As $\dim W = 5$, the only obstruction to extend the 
smooth structure on $\partial W$ to a smooth structure on $W$ is the Kirby-Siebenmann 
class ${\rm ks}(W,\partial W)\in H^4(W,\partial W;\bbz_2)$ (see \cite[Theorem~8.3.B]{Freedman-Quinn-Book}).
The image of ${\rm ks}(W,\partial W)$ under the isomorphism 
\beq{kscs}
H^4(W,\partial W;\bbz_2)\approx H_1(W,\bbz_2)\approx H_1(M;\bbz_2)\approx H^3(M;\bbz_2)
\eeq
coincides with ${\rm cs}(f)$ \cite[Remark~3.4.6]{Rudyak}.

Conversely, let $(W,M,N)$ be a smooth s-cobordism. If $\pi_1(M)$ is poly-(finite or cyclic), the topological
s-cobordism holds true (see the proof of \thref{T.n4}). Therefore, $W\ho M\times I \ (\rel M)$ and the topological
version of \exref{L.BMCf} makes $W$ homeomorphic $\rel M$
to the mapping cylinder $C_f$ of a homeomorphism $f\:M\to N$. 
Using \eqref{kscs}, one has ${\rm cs}(f)={\rm ks}(W,\partial W)=0$.  
\end{proof}

As $H^3(M;\bbz_2)\approx H_1(M;\bbz_2)$, one has the 
following corollary of \proref{P.CS}; it was proven by C.T.C. Wall~\cite{Wall4fourMfd}
when $M$ is simply connected, by a different method.

\begin{Corollary}\label{C.WallGen}
Let $M$ and $N$ be smooth closed manifolds of dimension $4$ which are homeomorphic.
Suppose that $H_1(M,\bbz_2)=0$. Then, $M$ and $N$ are smoothly s-cobordant.
\end{Corollary}

We are ready to prove Theorem~B of the introduction.

\begin{proof}[Proof of Theorem B]
Let $M$ and $N$ be smooth closed manifolds of dimension $4$ which are homeomorphic.
By \corref{C.WallGen}, there is a smooth h-cobordism $W$ between
$M$ and $N$. Such a cobordism is invertible (see \cite[Thm~4]{Stallings};
if $M$ is simply connected, then $W^\mun=\bar W$ \cite[Lemma~7.8]{RouSan}).
Thus $N\sdi M$ by \proref{P.invcob}.
\end{proof}

We now discuss a partial analogue to \proref{P.HL}, which was first proven by C.T.C~Wall in the simply connected case
\cite[Theorem~3]{Wall4fourMfd}. (See also Section~\ref{cstab}). 

\begin{Proposition}\label{P.HL4}
Let $M$ and $N$ be smooth closed connected manifolds of dimension $4$ which are $\bbr$-diffeomorphic.
Then, there exists $p\in\bbn$ such that
$$
M\,\sharp\, p(S^2\times S^2) \di N\,\sharp\, p(S^2\times S^2) \, .
$$
\end{Proposition}

\begin{proof}
A simplified handle decomposition 
$W = (M\times I) \pcirc W_2 \pcirc W_{3}$ as in 
\lemref{L.simformCob} is available but we do not know that the $2$-handles of $(W_2,M\times \{1\},M_2)$ 
are attached trivially (see \cite[Theorem~3 and its proof]{wallGeoConn}).
However, since $\pi_1(M)\approx\pi_1(W)$, the attaching map $\alpha\:S^1\times D^3\to M\times \{1\}$ of a $2$-handle of $W_2$
is homotopically trivial. As in the proof of \lemref{L.simformCob}, this implies, using an ambient isotopy of 
$M\times \{1\}$, that one may assume that $\alpha(S^1\times D^3)$ is contained in a disk. 
Also, $\alpha\:S^1=S^1\times \{0\}\to M\times \{1\}$ extends to an embedding $\alpha_-\:D^2\to M\times I$ and thus to an embedding
$\bar\alpha\:S^2\to W$. Since $\pi_2(M\times I)\to\pi_2(W)$ is an isomorphism, one can choose $\alpha_-$ so that 
$\bar\alpha$ is homotopically trivial. 

That $\alpha$ is attached trivially is thus equivalent to the triviality of the normal
bundle $\nu$ to $\bar\alpha$. As a vector bundle over $S^2$,  the Whitney sum $TS^2\oplus \nu$ is isomorphic
to $\bar\alpha^*TW$. The latter is trivial since $\bar\alpha$ homotopically trivial. As $TS^2$ is stably trivial, so is
$\nu$, which implies that $\nu$ is trivial since ${\rm rank\,}\nu > \dim S^2$. 
\end{proof}

Unlike in \proref{P.HL}, the torsion of an invertible cobordism between $M$ 
and $N$ only furnishes
a lower bound for the integer $p$ of \proref{P.HL4}, as seen by the case where $M$ and $N$ are
simply connected. An interesting question would be to find the minimal integer $p$ necessary to construct a given
invertible cobordism. Some results in the simply connected case may be found in \cite{Lawson2}.

We finish this section by considering the following problem which is important in view of \secref{S.classif}.

\begin{Problem}\label{Prob4}
Describe the set $\calb(M)$ for $M$ a smooth closed connected manifold of dimension $4$.
\end{Problem}

Only partial information is currently known about this problem. 
For instance, the map $\calt\:\calb(M)\to\wh(M)$ of \thref{T.hcob}, associating to an invertible cobordism $(W,M,N)$
its Whitehead torsion $\tau(W,M)$ is defined, and one has the following

\begin{Lemma}\label{L.surj4}
Let $M$ be a smooth closed connected manifolds of dimension $4$.
Then, the map $\calt\: \calb(M)\to\wh(M)$ is surjective.
\end{Lemma}
 
\begin{proof}
It is said in \cite[p.~102]{Freedman-Quinn-Book} that $\calt$ is surjective, based on ``the standard construction of
h-cobordisms'' with reference to \cite[p.~90]{RouSan}. But, when $n=4$, this standard construction 
for $\sigma\in\wh(M)$ only provides a cobordism $(W,M,N)$ such that the inclusion $M\hookrightarrow W$ 
is a homotopy equivalence with torsion $\sigma$. By Poincar\'e duality, one has $0=H^*(W,M;\bbz\pi)\approx H_*W,N;\bbz\pi)$,
where $\pi=\pi_1(W)\approx\pi_1(M)$. This proves that $W$ is a \dfn{semi-h-cobordism} from $N$, that
is to say that the inclusion $N\hookrightarrow W$ is homotopy equivalent to a Quillen plus-construction 
(see \cite{HV}); thus $i_*\:\pi_1(N)\to\pi$ is onto with perfect kernel $K$.

By \cite[Theorem~11.1A]{Freedman-Quinn-Book}, there exists a semi-s-cobordism $(W',N,N')$
with $\pi_1(M)\to\pi_1(W')$ onto with kernel $K$. Formula \eqref{duality} may be used here,
and thus $X=W\pcirc W'$ is an h-cobordism with $\tau(X,M)=\sigma$. 
As an h-cobordism between closed $4$-manifolds, $X$ is invertible \cite[Thm~4]{Stallings}.
\end{proof}

Some information is available on $\calb(M)$ when $M$ is simply connected. By \corref{C.WallGen}, the map $e$ of \eqref{D.mape}
may be replaced by a surjective map $e\:\calb(M)\to\calm(M)$, where $\calm(M)$ is the set of diffeomorphism  
classes of manifolds homeomorphic to $M$. This set may be infinite \cite{FinStern}, and so does $\calb(M)$. 
Let $\calm^0(M)$ be a set of representatives of $\calm(M)$. 
For $M$ oriented, one can precompose the bijection of $\lemref{L.partBM}$ by the surjective map
{\small
$$
\dia{
\coprod_{N\in\calm^0_{\rm or}(M)} \cob^{*,{\rm or}}(M,N)\Big/\aut^{\rm or}(N) \ar@{>>}[r] &
\coprod_{N\in\calm^0(M)} \cob^*(M,N)\Big/\aut(N) 
}
$$
}
where ``or'' stands for ``oriented''. Now, by \cite{Lawson2,Kreck1}, $\cob^{*,{\rm or}}(M,N)$ is in bijection with 
the set of isometries between the intersection forms of $M$ and $N$. 

\begin{Examples} \rm The above discussion implies the following facts.

\noindent
(1) {\it the case $M=S^4$.} \ The intersection form is trivial, so $\cob^{*,{\rm or}}(M,N)$ has one element
for each oriented homotopy sphere $N$. Note that $\cob^{*,{\rm or}}(M,-M)$ and $\cob^{*,{\rm or}}(M,M)$
are represented by the mapping cylinders of the identity or a reflection. By \lemref{L.mapscob1}, these 
cobordisms both represent $[S^4\times I[$ in $\calb(S^4)$.

\noindent
(2) {\it the case $M=\bbc P^2$.} \ The set $\cob^{*,{\rm or}}(M,-M)$ has one element and $\cob^{*,{\rm or}}(M,M)$ is empty. 

\noindent
(3) Results given in \cite[Proposition~8 and its proof]{Lawson2} imply, for instance, that 
$\cob^{*,{\rm or}}(M,M)\big/\aut^{\rm or}(M)$ is infinite for $M=\bbc P^2 \,\sharp\,k\,\overline{\bbc P^2}$ ($k\geq 9$).
\end{Examples}

The following result is a direct consequence of Example (1) above.

\begin{Proposition}\label{P.BS4}
The set $\calb(S^4)$ consists of one element if and only if 
the smooth Poincar\'e conjecture is true in dimension $4$.  \mancqfd
\end{Proposition}

\section{The case $n\leq 3$}\label{S.nthree}

We start with the proof of Theorem~C of the introduction (and then Theorem~A in low dimensions).

\begin{proof}[Proof of Theorem C]
There is only one closed manifold in dimension $1$, namely the circle. 
Closed surfaces are classified up to diffeomorphism by their fundamental group.
This proves  Theorem~C  when $n\leq 2$.

In dimension 3, let $M$ and $N$ be closed smooth orientable manifolds. Thanks
to the proof of the geometrization conjecture \cite{MorganTian2}, we know that 
$M$ and $N$ are geometric in the sense of Thurston. Therefore, if $M$ and $N$ 
are h-cobordant, a theorem of Turaev \cite[Theorem~1.4]{Turaev2} implies that 
they are homeomorphic, and hence also diffeomorphic by smoothing theory 
\cite[Theorem~6.4]{munkres}.
\end{proof}

\begin{Remark}\rm 
Theorem~C also follows from a theorem of Kwasik-Schultz which is interesting 
in itself: 
{\em an h-cobordism between geometric closed 3-dimensional manifolds $M$ and 
$N$ is an s-cobordism}
\cite[Theorem p.~736]{KwasikSchultz2}. One thus get a simple homotopy 
equivalence from $N$ to $M$, 
and such a map is homotopic to a diffeomorphism by \cite[Theorem~1]{Turaev} or 
\cite[Theorem~1.1]{KwasikSchultz2}.
\end{Remark}

\begin{Remark}\label{nonorthreemfds}\rm
We do not know if Theorem~C is true for closed non-orientable manifolds in dimension~3. The proof of 
\cite[Theorem~1.1]{KwasikSchultz2} uses the splitting theorem for homotopy equivalences
of \cite{HenLau}, which is wrong in general for non-orientable manifolds (see \cite{Hendriks}). 
Currently, a positive answer for the simplification problem for closed non-orientable
$3$-manifolds is only known for $P^2$-irreducible ones, i.e. irreducible (every embedded 2-sphere
bounds a 3-ball) and not containing any 2-sided $\rp{2}$. Such manifolds are indeed determined up to
diffeomorphism by their fundamental group \cite{Heil}.
\end{Remark}

We now turn our attention to the set $\calb(M)$.

\begin{Proposition}\label{P.BSurf}
Let $M$ be a smooth closed manifold of dimension $n\leq 2$. Then $\calb(M)$ contains one element. 
\end{Proposition}

\begin{proof}
Let $(W,M,N)$ be an h-cobordism with $n\leq 2$. 
We claim that $W\di M\times I$ if $n\leq 2$ (this implies that $W\di M\times I \ (\rel M)$). 
As an invertible cobordism is an h-cobordism by \proref{P.ihcob}, this will prove the proposition. 
The claim is obvious for $n=0$ and, for $n=1$, it follows from the classification of surfaces with boundary.
The case $n=2$ splits into three cases. We shall use the cobordisms 
$R_-=(D^3,\emptyset,S^2)$ and $R_+=(D^3,S^2,\emptyset)$.

\sk{1}\noindent (1) $M=S^2$. \ \rm
Let $(W,S^2,N)$ be an h-cobordism. By the classification of surfaces, there is a diffeomorphism $h\:S^2\to N$ and 
$\hat W = W\pcirc C_h$ is an h-cobordism from $S^2$ to itself, with $W\di \hat W \ (\rel S^2)$.
Then, $\Sigma^3=R_-\pcirc \hat W \pcirc R_+$ is a homotopy sphere, which is diffeomorphic to $S^3$
by Perelman's theorem  (\cite{Per, MorganTian}). Therefore, $\hat W$ is diffeomorphic to $S^3$ minus the interior of two smoothly embedded
$3$-disks, implying that $\hat W\di S^2\times I$.   

\sk{1}\noindent (2) $M=\rp{2}$. \ \rm
Suppose that $M=\rp{2}$. By composing $W$ with a mapping cylinder, we may assume that $N=\rp{2}$. 
Let $(\tilde W,\tilde M,\tilde N)$ be the universal covering of $W$, equipped with its involution $\tau$
(the deck transformation).
One has $\tilde M=\tilde N = S^2$, on which $\tau$ is the antipodal involution. As in (1),
form the closed $3$-manifold $\Sigma^3=R^3_-\pcirc \hat W \pcirc R^3_+$, diffeomorphic to $S^3$ 
by Perelman's theorem. The involution $\tau$ extends to an involution $\bar\tau$ on $\Sigma$ with
two fixed points $p_\pm$. By part~(c) of \proref{P.invcob2}, $W-\partial W\di M\times\bbr$. Therefore,
$\Sigma-\{p_\pm\}$ is equivariantly diffeomorphic to $S^2\times\bbr$ equipped with the involution
$\hat\tau(x,t)=(-x,t)$.

Hence, $(\Sigma,\bar\tau)$ is equivariantly homeomorphic to the suspension
of $(S^2,\tau)$. It follows that $\tilde W$ is equivariantly homeomorphic to $(S^2\times I,\hat\tau)$. 
Hence, $W\ho \rp{2}\times I$, implying that $W\di \rp{2}\times I$. 

\sk{1}\noindent (3) $\chi(M)\leq 0$. \ \rm
The discussion in \cite[pp.~97--99]{StallingsFib} implies that $W\di M\times I$.
\end{proof}

Much less is known about $\calb(M)$ when $M$ is a closed 3-manifold. When $M$ is orientable, we already used 
(in the proof of Theorem~C) the Kwasik-Schultz result that the Whitehead torsion map 
$\calt\: \calb(M) \to \wh(M)$ is identically zero. However, there are non-trivial s-cobordisms 
(see e.g.~\cite{CaSII,Kwasik1} for results and references).
The following question seems to be open.

\begin{Question}\label{h3inv} \rm
Is a smooth h-cobordism between closed 3-dimensional manifolds invertible?
\end{Question}

Here is a partial answer.

\begin{Proposition}\label{P.scobinv3}
Let $(W,M,N)$ be an s-cobordism between closed manifolds of dimension~$3$.
Suppose that $\pi_1M$ is poly-(finite or cyclic).
Then, $W$ is topologically invertible with $W^\mun=\bar W$.
\end{Proposition}

\begin{proof} {\it (following \cite[Lemma~7.8]{{RouSan}}).}
Consider $K=W\times I$ as a cobordism relative boundary from $M\times I$ to 
$(W\times \{0\})\pcirc (N\times I) \pcirc (\bar W \times \{1\})\di W\pcirc \bar W \ (\rel \partial)$.
Then $K$ is an s-cobordism. As $\dim (W\times I)=4$ and $\pi_1M$ is poly-(finite or cyclic),
the topological s-cobordism theorem implies that $W\di (M\times I)\times I \ (\rel M\times I\times \{0\})$.
Therefore, $W\pcirc \bar W \ho M\times I \ (\rel M)$. 
The same argument using the end $N\times I$ of $K$ gives that
$\bar W\pcirc W \ho N\times I \ (\rel M)$. 
\end{proof}

Here are two partial results when $M=S^3$.

\begin{Proposition}\label{P.BSk}
Let $(W,S^3,N)$ be a smooth h-cobordism. Then $W\ho S^3\times I \ (\rel S^3)$.   
\end{Proposition}

\begin{proof}
It is enough to prove the statement for $W$ a topological h-cobordism.  
By Perelman's theorem, there is a homeomorphism $h\:S^3\to N$ and 
$\hat W = W\pcirc C_h$ is an h-cobordism from $S^3$ to itself, with $W\ho \hat W \ (\rel S^3)$.
As in the proof of \proref{P.BSurf} (case of $M=S^2$), this implies that $W$ is the complement of two disjoint
tame $4$-disks in a homotopy sphere $\Sigma^4$. 
By Freedman's solution of the Poincar\'e conjecture \cite{Freedman}, $\Sigma\ho S^4$, which implies that 
$W\ho S^3\times I \ (\rel S^3)$. 
\end{proof}

\begin{Corollary}\label{C.BS3}
The following assertions are equivalent.
\begin{itemize}
\item[(a)] Any smooth h-cobordism $(W,S^3,N)$ is diffeomorphic to $S^3\times I$ relative $S^3$.
\item[(b)] The smooth Poincar\'e conjecture is true in dimension $4$. 
\end{itemize}
\end{Corollary}

\begin{proof}
The proof of \proref{P.BSk} shows that (b) implies~(a).
Conversely, let $\Sigma$ be a smooth homotopy $4$-sphere and let $K$ be a smooth submanifold of $\Sigma$ 
with $K\di D^4\dcup D^4$. Then $W=\Sigma - {\rm int\,}K$ is a smooth h-cobordism from $S^3$ to $S^3$.
If (a) is true, then $\Sigma \di D^4 \cup_h D^4$ for some self-diffeomorphism $h$ of $S^3$. 
Therefore, $\Sigma\di S^4$ \cite{Cerf}.
\end{proof}

We finish this section with the following open question.

\begin{Question}
If $(W,M,N)$ is an $h$-cobordism with $\dim M=3$, do we have $S^1\times W\di (S^1\times M)\times I \ (\rel S^1\times M)$ ? 
Note that the Whitehead torsion will vanish, by the product formula~\eqref{parProdFor}.  Hence this is true if 
$\dim M\geq 4$. 
\end{Question}

\section{Classifications of $\bbr$-diffeomorphisms}\label{S.classif}

In this section we examine the construction in \proref{P.invcob} further,
aiming for a full classification of $\bbr$-diffeomorphisms. 
The diffeomorphisms are classified under three levels of relations:
isotopy,  decomposability and concordance.

Let $M$ and $N$ be closed manifolds.
Let $\sdiff(M,N)$ be the set of $1$-diffeo\-morphisms from $N$ to $M$, endowed with the $\calc^\infty$-topology.
Thus, $\pi_0(\sdiff(M,N))$ is the set of isotopy classes of such $\bbr$-diffeomorphisms. 
For simplicity's sake, we restrict our attention to 
the subspace $\sdiff^+(M,N)$ of those $\bbr$-diffeomorphisms $f$ {\it preserving ends}, in the sense that 
$f(N\times[0,\infty))\subset M\times(r,\infty),\mathrm{\ for\ some\ }r\in \bbr$ (see also \remref{R.sdi+}).
As in \secref{S.icob}, $\aut(N)$ denotes the topological group of self diffeomorphims of $N$. 

In the proof of \proref{P.invcob}, an invertible cobordism $(A_f, j_M^r, f\pcirc j_N^s)$ (for suitable
 $r$ and $s$) 
was associated to $f\in\sdiff(N,M)$.
Consider its class $\cala_f$ in $Cob^*(M,N)$.  
Here is the fundamental observation leading to the other classification results. It is valid 
in all dimensions. 

\begin{Theorem}\label{T.ic0}
The correspondence $f \mapsto (A_f, j_M^r, f\pcirc j_N^s)$ 
induces a bijection 
$$
\cala: \pi_0(\sdiff^+(M,N))\fl{\approx}\cob^*(M,N) \, .
$$
Moreover, $\cala(\id_{M\times \bbr})=\bbun_M$, and if $f\in \sdiff^+(N, M)$
and $g\in\sdiff^+(P, N)$, then 
$\cala(f\pcirc g)= \cala(f)\pcirc\cala(g)$.
\end{Theorem}

Before we proceed, we remark that this gives a new interpretation of the category of invertible cobordisms.

\begin{Corollary} The category $\cob^*$ is isomorphic to the opposite of  the category where the objects are smooth
 manifolds and  the set of morphisms from $M$ to $N$ is $ \pi_0(\sdiff^+(M, N))$.
 \end{Corollary}

\begin{proof}[Proof of \thref{T.ic0}]
The proof involves several steps.

\sk{2}\noindent (1) {\it $\cala$ is well defined.} \
Let $f\:N\times\bbr\to M\times\bbr$ be an element of $\sdiff^+(N, M)$.
We use the notations of the proof of \proref{P.invcob}: $M_r=M\times \{r\}$, 
$N_u=N\times \{u\}$, $N'_u=f(N_u)$, etc.
Recall that, to define $A_f$, we choose $u$ and $r<s$ in $\bbr$
such that $N'_u\subset M\times (r,s)$.  The region from $M_r$ to $N'_u$
constitutes $A_f$ and that between $N'_u$ and $M_s$ constitutes the 
inverse $B_f$ of $A_f$. It is easy to check that 
$[A_f]=[A_f,j_M^r, f\pcirc j_N^u]\in\cob^*(M,N)$ does not depend on the 
choices of $r$ and~$u$. Consequently, we may assume that $u=0$.
\smallskip

Let $f_t\: N\times\bbr\to M\times\bbr$ ($t\in I$) be an isotopy between $f_0=f$ and $f_1=\hat f$.
Let $g_t$ be the restriction of $f_t$
to $N_0$. Since $N$ is compact, there exist $r<r_1<s_1<s$ in $\bbr$  such that 
$g_t(N_0)\subset M\times (r_1,s_1)$ for all $t$. By the isotopy extension 
theorem on $M\times [r,s]$ \cite[Theorem~1.3 in Chapter~8]{Hirsch}, 
there exists an ambient isotopy $F_t\:M\times\bbr\to M\times\bbr$, 
which is the identity outside $M\times [r_1,s_1]$ and
such that $g_t = F_t\pcirc g_0$. Using $r$ to define both $A_{f_0}$ and $A_{f_1}$, we see that 
$F_1$ provides a diffeomorphism from $A_f$ to $A_{\hat f}$ (relative $M_r$) such that 
$F_1\pcirc f\pcirc j_N^0 = \hat f\pcirc j_N^0$. Therefore, $[A_f]=[A_{\hat f}]$ in $\cob^*(M,N)$.

\sk{2}\noindent (2) {\it $\cala$ is surjective.} \  Let $A=(A,j_M, j_N)$ 
represent a class $\alpha\in\cob^*(M,N)$ and let $B=A^\mun$.
Composing infinitely many copies of $A\pcirc B$ as in \eqref{E.invcob}, we 
obtain a manifold $W$ together with two diffeomorphisms  
\beq{E.invcob2}
\dia{
M\times\bbr \ar[r]^(.59){g_M}_(.59)\approx & W  \ar@{<-}[r]^(.45){g_N}_(.45)\approx & N\times\bbr  \, ,
}
\eeq
Then $h=g_{\scr M}^\mun\pcirc g_{\scr N} \: N\times\bbr\to M\times\bbr$ is an element of $\sdiff^+(N,M)$ such that
$[A_h] = [A]$. Hence, $\cala(h)=\alpha$. 
 
\sk{2}\noindent (3) {\it $\cala$ is injective.} \  Let $f$ and $\hat f$ in $\sdiff^+(M,N)$
such that 
$\cala(f)=\cala(\hat f)$. Using observations in~(1), we can represent
$\cala(f)$ by $(A_f, j^0_M, f\pcirc j^u_N)$ and $\cala(\hat f)$ by 
$(A_{\hat f}, j^0_M, \hat f\pcirc j^{\hat u}_N)$, where we may assume that
$N'_u\subset \mathrm{int}\,A_{\hat f}$. In fact, after suitable isotopies 
of $f$ and $\hat f$ (by translations in the $\bbr$-direction) we may even 
assume that $u=\hat u=0$.  This means that we can write 
$[A_{\hat f}]=[A_f]\pcirc [K]$, where 
$[K]=[K, f\pcirc j^0_N,\hat f\pcirc j^0_N]$.
But if $\cala(f)=\cala(\hat f)$, the invertible cobordism $K$ must be equivalent to $\bbun_N$, 
i.\,e. there exists
a diffeomorphsim $F:N\times I\to K$ such that $F(x,0)=f(x,0)$ and 
$F(x,1)=\hat f(x,0)$ for all $x\in N$.

Now think of $F$ as an isotopy of embeddings  from $f\pcirc j^0_N$ to
$\hat f\pcirc j^0_N$. By the isotopy extension theorem there exists an 
ambient isotopy $H_t$ of $M\times \bbr$ such that $H_0=\id_{M\times\bbr}$ 
and  $H_1\pcirc f(x,0))=\hat f(x,0)$ for all $x\in N$.

Define $G:N\times \bbr \to N\times \bbr$ by $G=\hat f^{-1}\pcirc H_1\pcirc f$.
Then $G$ is a diffeomorphism such that $G(x,0)=(x,0)$ for all $x\in N$.
Considering $G$ and $\id_{N\times\bbr}$ as tubular neighborhoods of 
$N\times\{0\}$ in $N\times\bbr$, we see that $G$ is isotopic to the identity, 
by uniqueness \cite[Theorem~5.3 in Chapter~4]{Hirsch}.  It follows that
$\hat f$ is isotopic to $\hat f\pcirc G=H_1\pcirc f$, hence also to
$H_0\pcirc f=f$.\smallskip

\sk{2}\noindent (4) It is obvious that $\cala(\id_{M\times \bbr})=\bbun_M$, and it remains to 
prove the composition formula. Let  $f\in \sdiff^+(N, M)$
and $g\in\sdiff^+(P, N)$.
Start by choosing $u\in \bbr$ such that $f(N_u)\subset M\times (0,\infty)$,
and then $v\in \bbr $ such that $g(P_v)\subset N\times (u,\infty)$.  Then
the regions  $A_g$  between $N_u$ and $g(P_v)$,  $A_f$ between
$M_0$ and $f(N_u)$, and $A_{f\pcirc g}$ between $M_0$ and $f\pcirc g(P_v)$ can be used
to define $\cala(g)$, $\cala(f)$ and $\cala(f\pcirc g)$, respectively. In other words,
\begin{eqnarray*}
\cala(g) & = & [A_g, j^u_N, g\pcirc j^v_P]\\
\cala(f) & = & [A_f, j^0_M, f\pcirc j^u_N]\\
\cala(f\pcirc g) & = & [A_{f\pcirc g}, j^0_M, f\pcirc g\pcirc j_P]
\end{eqnarray*}
Now observe that we can write $A_{f\pcirc g}$ as $A_f\cup f(A_g)$, and 
consequently
\begin{eqnarray*} 
[A_{f\pcirc g}, j^0_M, f\pcirc g\pcirc j_P] & = & [A_f, j^0_M, f\pcirc j^u_N]
\pcirc [f(A_g), f\pcirc j^u_N, f\pcirc g\pcirc j_P]\\
& = &  [A_f, j^0_M, f\pcirc j^u_N]
\pcirc [A_g, j^u_N, g\pcirc j_P]\\
& = & \cala_f\pcirc \cala_g  \qedhere
\end{eqnarray*}
\end{proof}

We are now interested in another equivalence relation amongst  $\bbr$-diffeomorphism, using decomposability. 
A $\bbr$-diffeomorphism $f\in\sdiff^+(Q,Q')$ is called \dfn{decomposable} if there exists a diffeomorphism
$\varphi\:Q'\to Q$ such that $f$ is isotopic to $\varphi\times\id_\bbr$. 
Fix a manifold $M$ and 
consider pairs $(N,f)$ where $N$ is a smooth closed manifold and 
$f\:N\times\bbr\to M\times\bbr$ is a diffeomorphism. Two such pairs $(N,f)$ and $(\hat N,\hat f)$
are \dfn{equivalent} (notation: $(N,f)\sim (\hat N,\hat f)$)
if $f^\mun\pcirc\hat f$ is decomposable. The set of equivalences classes 
is denoted by $\cald(M)$. Note that $(N,f)$ is decomposable if and only if
$(N,f)\sim (M,\id)$.

\begin{Remark}\label{R.sdi+}\rm
The above definition of $\cald(M)$ is equivalent to the one presented in the introduction, where
$\bbr$-diffeomorphisms were not supposed to preserve ends. Indeed, $\sdiff^+(M,N)$ is a fundamental domain 
for the action of $\{\pm 1\}\approx \{\id_N\times\pm\id_\bbr\}$ by precomposition. 
\end{Remark}

\begin{Theorem}\label{T.ic}
Let $M$ be a smooth closed manifold.
The correspondence $(N,f) \mapsto [A_f[$ induces a bijection
$$
B\:\cald(M)\fl{\approx}\calb(M) \, .
$$
Moreover, $B(N,f)=[M\times I[$ if and only if $f$ is decomposable.
\end{Theorem}

\begin{proof}
Actually, the map $B$ is induced from the bijection $\cala$ of \thref{T.ic0}. 
As in \lemref{L.partBM}, let $\calm_n^0$ be a set of representatives of the diffeomorphism classes of closed manifolds of dimensioni $n$.
Consider the commutative diagram
\beq{T.ic-dia}
\dia{
\coprod_{N\in\calm_n^0}\pi_0(\sdiff^+(M,N)) \ar[r]^{\dcup\cala}_\approx \ar@{>>}[d]  &
\coprod_{N\in\calm_n^0} \cob^*(M,N) \ar@{>>}[d] \\
\coprod_{N\in\calm_n^0}\pi_0(\sdiff^+(M,N))\Big/\aut(N) \ar[r]^(.52){\dcup\bar\cala}_(.52)\approx \ar[d]^\approx &
\coprod_{N\in\calm_n^0} \cob^*(M,N)\Big/\aut(N) \ar[d]^\approx_\alpha \\
\cald(M) \ar[r]^{B}  &  \calb(M)
}
\eeq
The map $\dcup\cala$ is a bijection by \thref{T.ic0}. It intertwines the right-actions of $\aut(M)$ on $\cob^*(M,N)$ of \lemref{L.partBM}
with the ones defined on $\pi_0(\sdiff^+(M,N))$ by pre-composition using the inclusion 
$\aut(N)\to\sdiff^+(N,N)$ given by $\varphi\mapsto \varphi\times\id_\bbr$. The latter corresponds to the equivalence relation $\sim$
(note that $N\di \hat N$ if $(N,f)\sim (\hat N,\hat f)$).
That the map $\alpha$ is a bijection is the statement of \lemref{L.partBM}. Thus, the map $B$ is bijective.  
\end{proof}

\begin{Remark}\label{R.ic}\rm 
From part~(2) of the proof of \thref{T.ic0}, it follows that  $(N,f)\sim (N,g_{\scr M}^\mun\pcirc g_{\scr N})$,
where $g_{\scr M}$ and $g_{\scr N}$ are the diffeomorphisms 
constructed in in~\eqref{E.invcob2}. 
\end{Remark}

Thanks to \proref{P.BSurf}, \proref{P.BS4} and \corref{C.BS3}, \thref{T.ic} admits the following corollary.

\begin{Corollary}
Any diffeomorphism $f\:N\times\bbr\to M\times\bbr$ is decomposable if $\dim M\leq 2$. 
When $N=M=S^n$ with $n=3,4$, this is true if and only if the smooth Poincar\'e conjecture 
is true in dimension $4$.
\mancqfd
\end{Corollary}

The bijection $B\:\cald(M)\to\calb(M)$ of \thref{T.ic} may be composed with the map $\calt\: \calb(M)\to\wh(M)$,
associating to $W$ its Whitehead torsion $\tau(W,M)$. This gives a map $T\:\cald(M)\to\wh(M)$.
By \thref{T.hcob}, $\calt$ is a bijection when $n\geq 5$. Thus, \thref{T.ic} has the following corollary.

\begin{Corollary}\label{C1.ic}
Let $M$ be a smooth closed manifold of dimension $\geq 5$.
Then, the map $T\:\cald(M)\to\wh(M)$ is a bijection. 
Moreover, $T(N,f)=0$ if and only if $f$ is decomposable. 
\end{Corollary}

\corref{C1.ic} implies Theorem~D and Corollary~E of the introduction.
Another immediate consequence is the following:

\begin{Corollary}\label{C.xs1}
Let $M$ be a closed manifold and let $K$ be a closed manifold with Euler characteristic 0. 
The map $\cald(M)\to \cald(M\times K)$ given by product with the 
identity map on $K$ is trivial.
\end{Corollary}

In other words: if $f:N\times\bbr\stackrel{\approx}\to M\times\bbr$ is 
a diffeomorphism, then $f\times \id_{K}$ is
isotopic to a diffeomorphism of the form $h\times\id_{\bbr}$, where
$h$ is a diffeomorphism $N\times K\to M\times K$.

\begin{proof} The bijections $\cald(M)\approx \calb(M)\approx \wh(M)$ commute with
product with  $K$. The result then follows by the product formula for 
Whitehead torsion (\ref{parProdFor}).
\end{proof}

Diagram \eqref{T.ic-dia} gives a partition of $\cald(M)$ indexed by 
diffeomorphism classes of manifolds. Particularly interesting is the class
corresponding to $M$ itself, which via 
the bijection $\calb$  corresponds to the {\em inertial} cobordisms:
\begin{equation}\label{ibm}
\cali\calb(M)= \cob^*(M,M)/\Di(M) \approx \pi_0(\Di_+(M\times\bbr))/\Di(M).
\end{equation} 

\begin{Corollary}\label{C2.ic}
Let $M$ be a smooth closed manifold. The following assertions are equivalent.
\begin{itemize}
\item[(a)] Any automorphism $g\:M\times\bbr\to M\times\bbr$ is decomposable.
\item[(b)] $\cali\calb(M)$ has one element. 
\end{itemize}
Moreover, if $\dim M\geq 5$, Assertion~(b) may be replaced by
\begin{itemize}
\item[(b')] $I(M)=\{0\}$. 
\end{itemize}
\end{Corollary}

Manifolds $M$ such that $I(M)=\{0\}$ may be found in \exref{IneqWh}.

\begin{Example}\rm
Given two diffeomorphisms $f,g\:N\times\bbr\to M\times\bbr$, it is possible that $f^\mun\pcirc g$ is 
decomposable but not $g\pcirc f^\mun$. An example of this sort may be obtained using \corref{C2.ic}
and part~(3) of \exref{IneqWh}.
\end{Example}

In formula (\ref{ibm})  the second action is  right multiplication by the image of the group
homomorphism $\pi_0(\Di(M))\to \pi_0(\sdiff^+(M))$ induced by $\varphi\mapsto \varphi\times\id_\bbr$, and this corresponds
to the map (also homomorphism!) $\pi_0(\Di(M))\to \cob^*(M,M)$ given by 
$f\mapsto C_{f^{-1}}$ (mapping cylinder). As seen in \exref{ex.conc}, this map is not 
injective, but has as kernel the isotopy classes of diffeomorphisms 
{\em concordant} to the identity.  This leads to the following result, first
proved by W. Ling in the topological category \cite{LingW}.  Let 
$C(M)=\{f\in \Di(M\times I)\,|\, f|M\times\{0\}=\id\}$ be the space of 
concordances of $M$.  Then evaluation on $M\times\{1\}$ gives rise to a 
fibration (over a union of components) $ C(M)\to \Di(M)$, with fiber
$\Di(M\times I,\mathrm{rel}\,M\times\partial I)$.

\begin{Proposition}\label{P.ling}
The long, exact sequence of homotopy groups of this fibration ends as follows:
{\small
$$
\dia{
\cdots\ar[r] & 
 \pi_0(C(M)) \ar[r]&
\pi_0(\Di(M))\ar[r]&
\pi_0(\sdiff^+(M))\ar@{->>}[r] &
\cali\calb(M)} 
$$
}
\end{Proposition}

\begin{proof} The last map in the ordinary long exact sequence  is the homomorphism $\pi_0(C(M))\to \pi_0(\Di(M))$ 
with image the
set of isotopy classes of diffeomorphisms concordant to the identity, which we just saw is also the kernel of 
the homomorphism $\pi_0(\Di(M))\to \pi_0(\sdiff^+(M))$.  The last map is just the quotient map onto the
set of left cosets.
\end{proof}, 

\begin{Remark}{\rm
It is known that $\Di(M\times\bbr)$ is a non-connected delooping of 
$\Di(M\times I,\mathrm{rel}\,M\times\partial I)$. (See e.\,g. \cite{WW1}.)
\proref{P.ling} gives more information on components.}
\end{Remark}

We now use the relation of concordance to give a classification
of $1$-diffeo\-morphisms which is coarser than isotopy. 
Following the pattern above, we first say that a $\bbr$-diffeomorphism
$f\in\sdiff^+(Q',Q)$ is \dfn{c-decomposable} if there exists 
a diffeomorphism 
$\varphi\:Q'\to Q$ such that $f$ is concordant to $\varphi\times id_\bbr$. 
Then $(\hat N,\hat f)$ and $(N,f)$ are called
\dfn{c-equivalent} (notation: $(\hat N,\hat f)\sim_c (N,f)$) if 
$f^\mun\pcirc\hat f$ is c-decomposable. 
Of course, $(\hat N,\hat f)\sim (N,f)$ implies $(\hat N,\hat f)\sim_c (N,f)$;
therefore, the set $\cald_c(M)$  of these c-equivalences classes is a quotient of $\cald(M)$.

Using the the bijection $B$ of \thref{T.ic}, the equivalence relation $\sim_c$ on $\cald(M)$ 
may be transported to $\calb(M)$, giving rise to an equivalence relation on 
$\calb(M)$,  also denoted $\sim_c$.  We want to prove that $\sim_c$ can
be described in terms of the relation of {\em concordance of invertible 
cobordisms}, defined in \remref{concordance}.

Recall again the partition

$$\dia{
\coprod_{N\in\calm_n^0} \cob^*(M,N)\Big/\Di(N) \ar[r]^(0.72){\alpha}_(0.72){\approx} & \calb(M) . 
}$$
of \lemref{L.partBM}.  In \remref{concordance} the relation of (invertible) concordance is defined on each set $\cob^*(M,N)$, 
and the action of $\Di(N)$ descends to the set of concordance classes $\overline{\cob^*}(M,N)$. 
Set 
\beq{E.relcdef}
\calb_c(M)=\coprod_{N\in\calm_n^0} \overline{\cob^*}(M,N)\Big/\aut(N).
\eeq

Like \thref{T.ic}, the following result is valid in all dimensions. 

\begin{Theorem}\label{T.ic2}
Let $M$ be a smooth closed manifold.
Then, the bijection $B\:\cald(M)\to\calb(M)$ of \thref{T.ic}
descends to a bijection 
$$
B_c\:\,\cald_c(M)\fl{\approx}\calb_c(M) \, .
$$
\end{Theorem}

\begin{proof}
Given part~(i) of the proof of \thref{T.ic}, in order to define $B_c$, we just need to prove that
when $f,\hat f\:N\times\bbr\to M \times\bbr$ are concordant, then 
$[A_f]=[A_f,j_M^r\dcup f\pcirc j_N^0]$ and $[A_{\hat f}]=[A_{\hat f},j_M^r\dcup \hat f\pcirc j_N^0]$
represent the same class in $\overline{\cob^*(M,N)}$. 
Let $F\:I\times N\times\bbr\to I\times M\times\bbr$ be a concordance between $f$ and $\hat f$. 
The construction of $A_f$, $B_f$, $A_{\hat f}$ and $B_{\hat f}$ may be done globally in $I\times N$
and $I\times M$. This would provide cobordisms $A_F$ between $A_f$, $A_{\hat f}$, and
$B_F$ between $B_f$, $B_{\hat f}$ which are inverse of one another, which is what we need.

The map $B_c$ is thus well defined. It is surjective, since $B$ is.
To prove that $B_c$ is injective, we use a relative version of the proof of 
surjectivity in \thref{T.ic0}.  Let $(N,f)$ and $(\hat N,\hat f)$ represent classes in $\cald(M)$ 
such that $B(N,f)\sim_cB(\hat N,\hat f)$. Since the relation 
$\sim_c$ preserves $\cob^*(M,N)$, this means that there is a diffeomorphism 
$\gamma\:\hat N\to N$ such that
$B(\hat N,\hat f)\sim B(N,\hat f\pcirc(\gamma\times \id_\bbr))$. This permits us to 
assume that $\hat N = N$. In this case, $B(N,f)$ and $B(N,\hat f)$ are represented by 
$[A_f]$ and $[A_{\hat f}]$ in $\cob^*(M,N)$ such that
$[A_f]$ is invertibly concordant to [$A_{\hat f}]\,\beta$ for some $\beta\in \aut(N)$. 
Using again that $(N,\hat f)\sim (N,\hat f\pcirc(\beta\times \id_\bbr))$, 
we may assume that $[A_{\hat f}]=[A_f]$ in $\overline{\cob^*}(M,N)$.
 
Let $[K]$ be a concordance between $A_f$ and $A_{\hat f}$, with inverse $[L]$ from  $[B_f]$ and $[B_{\hat f}]$.
Let $K_i$ and $L_i$ ($i\in\bbz$) be copies of $K$ and $L$. As in~\eqref{E.invcob}, we form the manifold
\beq{E.invcobxI}
\begin{array}{rcl}
X &=& \cdots\pcirc\, (K_i\pcirc L_i)\, \pcirc \,(K_{i+1}\pcirc L_{i+1}) \,\pcirc \cdots \\ &=& 
\cdots\pcirc\, (L_i\pcirc K_{i+1})\,\pcirc\,  (L_{i+1}\,\pcirc K_{i+2})        \pcirc\cdots
\end{array}
\eeq
Using convenient diffeomorphisms $K_i\pcirc L_i\di I\times M\times I$ and 
$L_i\pcirc K_{i+1}\di I\times N\times I$, one gets, as in~\eqref{E.invcob2}, two diffeomorphisms
\beq{E.invcob2xI}
\dia{
I\times M\times\bbr \ar[r]^(.65){G_M}_(.65)\approx & X \ar@{<-}[r]^(.42){G_N}_(.42)\approx & 
I\times N\times\bbr
}
\eeq

The diffeomorphism $F=G_{M}^\mun\pcirc G_N\: I\times N\times\bbr\to I\times M\times\bbr$ 
restricts to diffeomorphisms $F_i\:\{i\}\times N\times\bbr\to \{i\}\times M\times\bbr$ ($i=0,1$)
and $F$ constitutes a concordance between $F_0$ and $F_1$. Therefore, $(N,F_0)\sim_c (N,F_1)$.
By \remref{R.ic}, one has $(\{0\}\times N,F_0)\sim (N,f)$ and $(\{1\}\times N,F_1)\sim (N,\hat f)$. 
Therefore, $(N,f)\sim_c (N,\hat f)$, which proves the injectivity of $B_c$. 
\end{proof}

We now compute $\calb_c(M)$ when $\dim M \geq 5$, using the bijection 
$\calt\: \calb(M) \to \wh(M)$ of \thref{T.hcob}.
As in \exref{IneqWh}, we consider the subgroup $\caln(M)$ of $\wh(M)$ defined by
$$
\caln(M) = \{\tau+ (-1)^{n}\bar\tau\mid \tau\in\wh(M)\}  \, ,
$$
using the involution $\tau\mapsto\bar\tau$ of~\eqref{defInvol}.

The following result  now follows easily from the discussion at the end of
\secref{S.icob}:

\begin{Proposition}\label{P.iccbar}
Let $M$ be a smooth closed manifold of dimension $n\geq 5$.
Then, the bijection $\calt\:\calb(M)\to\wh(M)$ of \thref{T.hcob}
descends to a bijection 
$$
\calt_c\:\,\calb_c(M)\fl{\approx}\wh(M)/\caln(M) \, . 
$$
\end{Proposition}

\begin{proof}
That $\calt_c$ is well-defined follows from \lemref{L.WhConc}, and 
surjectivity is trivial. Assume now that the torsions of
two invertible cobordisms   $(W,j_M,j_N)$ and $(W',j'_M,j'_N)$
satisfy the equation 
$\tau(W',j'_M)-\tau(W,j_M)=\sigma+(-1)^n\bar\sigma$ for some 
$\sigma\in\wh(M)$, where $n=\dim M$.

There is  a relative h-cobordism $(X,W,V)$
with $\tau(X,W)= j_M{}_*(\sigma)$, where  $V$ is another h-cobordism from
$j_M(M)$ to $j_N(N)$. By \proref{P.ihcob} $X$ and $V$ are both invertible,
and by \lemref{L.WhConc} we have 
$\tau(V,j_M)=\tau(W',j_M)$. By uniqueness of Whitehead torsion,
$[W,j_M[=[V,j_M[\in\calb(M)$. 
\end{proof}

\thref{T.ic2} together with \proref{P.iccbar} implies the following corollary.

\begin{Corollary}\label{C1.icc}
Let $M$ be a smooth closed manifold of dimension $\geq 5$.
Then, the bijection $T\:\cald(M)\to\wh(M)$ of \corref{C1.ic}
descends to a bijection 
$$
T_c\:\cald_c(M)\to\wh(M)/\caln(M) \, .
$$
Moreover, $T_c(N,f)=0$ if and only if $f$ is c-decomposable. 
\mancqfd
\end{Corollary}

Recall the inclusion $\caln(M)\subset I(M)$, which is not an equality in 
general. \corref{C1.icc} implies the following result.

\begin{Corollary}
Let $M$ be a smooth closed connected manifold of dimension $\geq 5$.
The following assertions are equivalent.
\begin{itemize}
\item[(a)] Any automorphism $g\:M\times\bbr\to M\times\bbr$ is c-decomposable.
\item[(b)] $\caln(M)=I(M)$.  \mancqfd
\end{itemize}
\end{Corollary}

\begin{Example}\rm 
Let $M$ be a smooth closed connected manifold of dimension $n\geq 5$ such that $\pi=\pi_1(M)$ is
cyclic of order $5$ with generator $t$. Then, $\wh(M)\approx\bbz$ generated by 
$\sigma=(1-t-t^4)\in GL_1(\bbz\pi)$ \cite[Example~6.6]{MilnorWT}. 
We see that $\sigma=\bar\sigma$, so the involution on $\wh(M)$ is trivial. Therefore,
\begin{itemize}
\item if $n$ is odd, $\caln(M)=0$ and then $\cald(M)=\cald_c(M)\approx\bbz$; thus concordance 
implies isotopy for $\bbr$-diffeomorphisms with range $M\times\bbr$;
\item if $n$ is even, then $\cald(M)\approx\bbz$ and $\cald_c(M)\approx\bbz_2$. Thus, 
for diffeomorphisms with range $M\times\bbr$, there are infinitely
many isotopy classes within the same concordance class.
\end{itemize}
\end{Example}

\section{Miscellaneous}\label{S.Rem}

\begin{ccote}\label{fakeRfour}\rm This paper deals with $\bbr$-diffeomorphisms between closed manifolds.
For open manifolds, there is a long story of negative answers to the $\bbr$-simplification problem, 
starting with the earlier example of J.H.C. Whitehead~\cite[p.~827]{WheadMan}. 
There is also the famous {\it Whitehead manifold} which is $\bbr$-diffeomorphic but not homeomorphic to $\bbr^3$ 
(see, e.g. \cite[pp.~61--67]{deRham}).
The most striking example is given by the uncountable family of fake $\bbr^4$'s (see e.g. \cite{GompfMena}),
which are all $\bbr$-diffeomorphic, since there is only one smooth structure on $\bbr^5$ \cite[Corollary~2]{StallingsRn}.
\end{ccote}

\begin{ccote}\label{hinote} Historical note. \rm \
As seen in Sections~\ref{S.icob}--\ref{S.nthree}, Theorem~A of the introduction is equivalent to 
the smooth h-cobordism theorem of Smale \cite{SmalePC} for $n\geq 5$, 
and to the topological h-cobordism theorem of Freedman for $n=4$ \cite{Freedman-Quinn-Book}.
For $n=3$ it is a consequence of Perelman's proof of the Poincar\' e conjecture (see \cite{MorganTian}).
There is no known proof not using these formidable results for which three Field medals were awarded.
Finally, for $n=2$, Theorem~A requires the classification of surfaces, a classical but not trivial result. 
Note that the simplification problem is a geometric form of the problem of recognizing the diffeomorphism type
of a smooth closed manifold by its homotopy type, one of the most important problems 
of algebraic topology, going back to the birth of the subject (see e.g. \cite[\S\,5.1]{HauBook}). 
\end{ccote}

\begin{ccote}\label{kstab} $ \bbr^k$-diffeomorphisms \rm were introduced by B.~Mazur~\cite{Mazur} under the name of {\it $k$-equivalences}.  
Note that a diffeomorphism $f\: M\times\bbr^k\to N\times\bbr^k$ induces a stable tangential homotopy equivalence
(still called $f$) from $M$ to $N$. The \dfn{thickness} of such a stable tangential homotopy equivalence $f$ is the minimal $k$
for which $f$ is induced by an $\bbr^k$-diffeomorphism \cite{KwasikSchultz4}. This thickness is $\leq \dim M + 2$ \cite[Theorem~1]{Mazur}.
For more results, see e.g. \cite{KwasikSchultz4,JaKwaRP,KwaRok}.
\end{ccote}

\begin{ccote}\label{Psimpl} \rm
The $P$-simplification problem has been studied for $P$ a sphere, a torus or a surface.
See e.g. \cite{HiMiRo} for results and several references, and also \remref{R.Sunsimpl}. For more recent results, see e.g. 
\cite{KwaRok, JaKwaRP, KwasikSchultz3, KwaRok2}. 
\end{ccote}

\begin{ccote}\label{cstab}  Stable diffeomorphisms. \rm 
Two closed manifolds $M,N$ of dimension $2n$ are called \dfn{stably diffeomorphic} in the literature if 
$M\,\sharp\,p(S^n\times S^n)\di N \,\sharp\,p(S^n\times S^n)$
for some integer $p$.  Thus Corollary 4.3 and Proposition 5.6 say that $\bbr$-diffeomorphism implies stable diffeomorphism.
The stable diffeomorphism class of a manifold may be detected by cobordisms invariants, as initiated by M.~Kreck~\cite{KreckSD}. 
For recent results and many references, see~\cite{KMPT}.
\end{ccote}

\begin{ccote}\label{spaceforms} Generalized spherical spaceforms. \ \rm
A manifold is a \dfn{generalized spherical spaceform} if its universal covering is a homotopy sphere.
Let $M$ and $N$ be diffeomorphic generalized spherical spaceforms of dimension $\geq 5$. Then Kwasik and Schultz
have proved that any h-cobordism between $M$ and $N$ is trivial \cite{KwasikSchultz}.
This implies that $I(M)=0$ and, thus, $\bbr$-diffeomorphism implies diffeomorphism.
\end{ccote}

\begin{ccote}\rm
In general relativity, the $\bbr$-simplification problem has natural applications to  the classification 
of Cauchy surfaces in globally hyperbolic spacetimes. (See \cite{Torres} for results and references).
\end{ccote}

\sk{5}\noindent
\begin{minipage}[t]{6cm}
Jean-Claude HAUSMANN\\
Department of Mathematics\\
University of Geneva\\
CH-1211 Geneva 4,
Switzerland\\
Jean-Claude.hausmann{@}unige.ch
\end{minipage}
\hfill
\begin{minipage}[t]{4cm}
Bj\o rn JAHREN\\
Department of Mathematics\\
University of Oslo\\
0316 Oslo,
Norway\\
bjoernj@math.uio.no
\end{minipage}

\end{document}